\documentclass[reqno,a4paper]{amsart}

\usepackage[italian,english]{babel}
\usepackage{eucal,amsfonts,amssymb,amsmath,amsthm,epsfig,mathrsfs}
\usepackage{color}
\usepackage[dvipsnames]{xcolor}

\usepackage{amscd,amsxtra}
\usepackage{enumerate}
\usepackage{latexsym}
\usepackage{bm}

\allowdisplaybreaks

\newcounter{ipotesi}

\Alph{ipotesi}
 \makeatletter \@addtoreset{equation}{section}

\makeatother \makeatletter

\newtheorem{thm}{Theorem}[section]
\newtheorem{hyp}[thm]{Hypotheses}{\rm}
{\rm}

\newtheorem{prop}[thm]{Proposition}
\newtheorem{defi}[thm]{Definition}
\newtheorem{rmk}[thm]{Remark}{\rm}
\newtheorem{example}[thm]{Example}

\newcounter{parentenv}

\newcommand{\R}{{\mathbb R}}

\newcommand{\N}{{\mathbb N}}

\newcommand{\Rd}{\mathbb R^d}

\newcommand{\Rm}{\mathbb R^m}

\newcommand{\supp}{{\rm{supp}}\,}
\newcommand{\T}{{\bf T}}
\newcommand{\Res}{{\bf R}}

\newcommand{\e}{{\bf e}}
\newcommand{\g}{{\bf g}}

\newcommand{\f}{{\bf f}}
\newcommand{\uu}{{\bf u}}
\newcommand{\vv}{{\bf v}}
\newcommand{\A}{\boldsymbol{\mathcal A}}

\allowdisplaybreaks

\begin{document}

\title[Invariant measures for systems of Kolmogorov equations]{Invariant measures for systems of Kolmogorov equations}
\thanks{Work partially supported by the INdAM-GNAMPA Project 2017 ``Equazioni e sistemi di equazioni di Kolmogorov in dimensione finita e non''.}
\author[D. Addona, L. Angiuli, L. Lorenzi]{Davide Addona, Luciana Angiuli and Luca Lorenzi}
\address{D.A.:Dipartimento di Matematica e Informatica, Universit\`a degli Studi di Ferrara, via N.  Machia\-velli 30, I-44121 Ferrara, Italy}
\address{L.A.: Dipartimento di Matematica e Fisica ``Ennio De Giorgi'', Universit\`a del Salento, Via per Arnesano, I-73100 LECCE (Italy)}
\address{L.L.: Dipartimento di Scienze Matematiche, Fisiche e Informatiche, Plesso di Matema\-ti\-ca e Informatica, Universit\`a degli Studi di Parma, Viale Parco Area delle Scienze 53/A, I-43124 Parma, Italy}
\email{d.addona@campus.unimib.it}
\email{luciana.angiuli@unisalento.it}
\email{luca.lorenzi@unipr.it}

\keywords{Systems of Kolmogorov equations, Unbounded coefficients, Gradient estimates, Systems of invariant measures, Asymptotic behaviour.}
\subjclass[2000]{35B40, 35B41, 35K45, 37L40, 47A35.}

\begin{abstract}
In this paper we provide sufficient conditions which guarantee the existence of a system of invariant measures
for semigroups associated to systems of parabolic differential equations with unbounded coefficients. We prove that these measures are absolutely continuous with respect to the Lebesgue measure and  study some of their main properties. Finally, we show that they characterize the asymptotic behaviour of the semigroup at infinity.
\end{abstract}

\maketitle

\section{Introduction}
In this paper we prove the existence of systems $\{\mu_i: i=1,\ldots,m\}$ of finite signed Borel measures which
satisfy the equality
\begin{align}
\label{freddo}
\sum_{i=1}^m\int_{\Rd}(\T(t)\f)_id\mu_i=\sum_{i=1}^m\int_{\Rd}f_id\mu_i, \qquad\;\, i=1,\ldots,m,
\end{align}
for any $t\geq0$ and $\f\in B_b(\Rd;\Rm)$ (the space of vector-valued bounded Borel functions $\f:\Rd\to\Rm$) or, equivalently, for any
$\f\in C_b(\Rd;\Rm)$ (the subspace of $B_b(\Rd;\Rm)$ consisting of continuous functions). Here, $\{\T(t)\}_{t\geq0}$ (from now on simply denoted by $\T(t)$) is the semigroup of bounded linear operators on $B_b(\Rd;\Rm)$, associated to the vector-valued differential operator ${\A}$, defined on smooth functions
${\bm\zeta}=(\zeta_1,\ldots,\zeta_m)$ by
\begin{align}
(\A\boldsymbol{\zeta})_j(x)
:= & \sum_{h,k=1}^dq_{hk}(x)\frac{\partial^2\zeta_j}{\partial x_hx_k}(x)+\sum_{k=1}^db_k(x)\frac{\partial\zeta_j}{\partial x_k}(x)+\sum_{k=1}^d\sum_{i=1}^m(B_k(x))_{ji}\frac{\partial\zeta_i}{\partial x_k}(x)
\label{finale}
\end{align}
for any $x\in\Rd$ and $j=1,\ldots,m$, under suitable assumptions on its coefficients. Formula \eqref{freddo} seems the natural vector-valued counterpart of the
invariant measure of the scalar case. A probability measure is called invariant for a Markov semigroup $\{T(t)\}_{t\ge 0}$ (from now on simply denoted by $T(t)$) associated in $B_b(\Rd)$ to
an elliptic operator (with unbounded coefficients) ${\mathcal A}$ if
\begin{eqnarray*}
\int_{\Rd}T(t)fd\mu=\int_{\Rd}fd\mu,\qquad\;\,t>0,\;\,f\in B_b(\Rd).
\end{eqnarray*}
For this reason, we call {\it system of invariant measures} for $\T(t)$ any family $\{\mu_i: i=1,\ldots,m\}$ of finite measures which satisfy \eqref{freddo}.

In the scalar case, under quite mild (algebraic) conditions on the coefficients of the operator ${\mathcal A}$, there exists a unique invariant measure $\mu$ and $\mu$ is equivalent to the Lebesgue measure (see \cite{metafune-pallara-wacker1}).
This measure plays an essential role in the analysis of the semigroup $T(t)$. Indeed, if the coefficients of ${\mathcal A}$ are unbounded, then the
$L^p$-spaces with respect to the measure $\mu$ (say $L^p(\Rd,\mu)$) are the best $L^p$-setting where to study the semigroup $T(t)$
(see e.g., \cite{FLun-2006, lorenzi-lunardi-elliptic, lunardi-ontheO-U, metafune-pruss-rhandi-schnaubelt}). As it is shown in \cite{pruss-rhandi-schnaubelt}, the usual Lebesgue $L^p$-spaces are not, in general, a good setting for the semigroup $T(t)$,
unless (restrictive) assumptions are prescribed on the coefficients of the associated elliptic operator (see also \cite{ALP} for the vector-valued case).
As a matter of fact, the measure $\mu$ is not explicit in general. In any case, both local and global regularity results for its density $\rho$ with respect to the Lebesgue measure are known in many cases
(see e.g., \cite{bogkroroc-00, newbook, MPR-globalprop05}). The typical way to study the regularity of $\rho$
is to see it as a distributional solution of the adjoint equation ${\mathcal A}^*\rho=0$.

The relevance of the invariant measure $\mu$ lies also on the fact that they allow to characterize the asymptotic behaviour as $t$ tends to $+\infty$ of the
semigroup they are associated to. Indeed, for any $f\in L^p(\Rd,\mu)$,  $T(t)f$ converges in $L^p(\Rd,\mu)$ to the average of $f$ with respect to the measure $\mu$,
as $t$ tends to $+\infty$ (see e.g., \cite{daprato-z2, newbook}). The convergence is also local uniform in $\Rd$ if $f$ is bounded and continuous (see e.g., \cite{LorLunSch16Str}).

For semigroups associated to systems of elliptic equations, as the case that we are considering here, the situation is much more complicated and the picture is still far to be completely clear. One of the most typical feature of the scalar case is the positivity of the semigroup $T(t)$, which follows from a variant of the classical maximum principle. This property and the ergodicity of $T(t)$ imply in a straightforward way that, whenever it exists, $\mu$ is a positive measure since
\begin{eqnarray*}
\int_{\Rd}fd\mu=\lim_{t\to +\infty}\frac{1}{t}\int_0^t(T(s)f)(0)ds,\qquad\;\,f\in C_b(\Rd).
\end{eqnarray*}
As it is well known, already in the case of bounded and smooth
enough coefficients, in general the semigroups associated to systems of elliptic operators do not preserve the cone of functions $\f$ with all nonnegative components. Indeed, \cite{Ots88Ont} shows that the semigroup associated to the operator $\A_0$ (with smooth and bounded coefficients), defined on smooth functions $\bm\zeta$ by
\begin{align*}
(\A_0{\bm\zeta})_j(x):= \sum_{h,k=1}^dq_{hk}(x)\frac{\partial^2\zeta_j}{\partial x_hx_k}(x)+\sum_{k=1}^d\sum_{i=1}^m
(B_k(x))_{ji}\frac{\partial\zeta_i}{\partial x_k}(x)+\sum_{i=1}^mC_{ji}\zeta_i(x)
\end{align*}
for any $x\in\Rd$ and $j=1,\ldots,m$, is positive if and only if the drift terms of $\A_0$ are diagonal, i.e.,
$(B_k)= b_kI_m$ for any $k=1,\ldots,d$, and the potential matrix $C$ has nonnegative elements
outside the main diagonal (see also \cite{KreMaz12Max}).

To the best of our knowledge, the first paper which deals with systems of invariant measures for systems of Kolmogorov equations is 
\cite{AAL}, where the case of {\it weakly coupled} systems with a potential term is considered, i.e., in that situation the operator $\A$ is defined, on smooth enough
functions $\bm\zeta$, by
\begin{eqnarray*}
(\A\boldsymbol{\zeta})_j(x)
=\sum_{h,k=1}^dq_{hk}(x)\frac{\partial^2\zeta_j(x)}{\partial x_hx_k}+\sum_{k=1}^db_k(x)\frac{\partial\zeta_j}{\partial x_k}(x)+\sum_{i=1}^m(C(x))_{ji}\zeta_i(x)
\end{eqnarray*}
for any $j=1,\ldots,m$, $x\in\Rd$. Under suitable conditions on $C$ which, in particular, imply that the associated semigroup $\T(t)$ is bounded, in \cite{AAL} we prove that the semigroup $\T(t)$ also preserves the cone of nonnegative functions and this makes the analysis easier. In particular, we are able to characterize all the systems of invariant measures in terms of the invariant measure of
the scalar semigroup $T(t)$ associated to the operator ${\mathcal A}={\rm Tr}(QD^2)+\langle {\bf b},\nabla \rangle$ in the space of bounded and continuous functions over $\Rd$.

This paper represents the second step in a better understanding of systems of invariant measures, its analogies and differences with the invariant measure of the scalar case.
Motivated by the scalar case and also by the results in \cite{AAL} we would like to define a system of invariant measure through the limit (in a suitable sense)
\begin{equation}
\lim_{t\to 0}\frac{1}{t}\int_0^t(\T(s)\f)ds,
\label{limit}
\end{equation}
when $\f:\Rd\to\Rm$ is an arbitrary bounded and continuous function.
The first problem that we have to face is that, in the scalar case, $T(t)$ is a bounded semigroup.
In general, this is no more true for semigroups associated to systems of Kolmogorov equations coupled up to the first-order  (see Remark \ref{rem-2.1}).
As a it is shown in \cite{AALT} the semigroup $\T(t)$ admits the integral representation
\begin{eqnarray*}
(\T(t)\f)_i(x)=\sum_{j=1}^m\int_{\Rd}f_j(y)p_{ij}(t,x,dy),\qquad\;\,\f\in B_b(\Rd;\R^m),\;\,i=1,\ldots,m,
\end{eqnarray*}
for any $x\in\Rd$, where each $p_{ij}(t,x,dy)$ is a signed measure. In Proposition \ref{lem-2.1} we show that the boundedness of $\T(t)$ is
equivalent to the boundedness of the family of measures $\{|p_{ij}|(t,x,dy): t>0, x\in\Rd\}$ for any $i,j=1,\ldots,m$, where $|p_{ij}|(t,x,dy)$ denotes the total variation
of the measure $p_{ij}(t,x,dy)$.
Nevertheless, even if this condition is not satisfied, under suitable conditions and using the pointwise gradient estimate in Proposition \ref{bundes} we prove that, for each $x_0\in\Rd$ and each $\f\in B_b(\Rd;\Rm)$, the function $(\T(\cdot)\f)(x_0)$
is bounded in $(0,+\infty)$ (see Theorem \ref{parodontax}). This fundamental result allows us to prove that the limit in \eqref{limit} exists in the sense that the function
$x\mapsto \frac{1}{t}\int_0^t(\T(s)\f)(x)ds$ converges locally uniformly in $\Rd$. The limit $\g_{\f}$, which is a continuous, but {\it a priori} an {\em unbounded} function, has  a controlled growth at infinity, and this property allows us to apply the semigroup $\T(t)$ to such a function $\g_{\f}$. It turns out that $\T(t)\g_{\f}=\g_{\f}$ for any $t>0$, i.e., $\g_{\f}$ is a fixed point of the semigroup $\T(t)$. Using again the pointwise gradient estimates, we can then conclude that $\g_{\f}$ is a constant function.
This allows us to define $m$ systems of invariant measures for the semigroup $\T(t)$ (say, $\{\mu_j^i: j=1,\ldots,m\}$, $i=1,\ldots,m$).

We then exploit some properties of the above systems of invariant measures. We show that each measure $\mu_j^i$ is absolutely continuous with respect to the Lebesgue measure
and prove some regularity and integrability properties of the density of their total variations with respect to the Lebesgue measure (see Proposition \ref{giampaolo-co} and Theorem \ref{capitano}). We also prove that a suitable unbounded function $\zeta$ (which is a power of the Lyapunov function of the scalar operator ${\mathcal A}={\rm Tr}(QD^2)+\langle {\bf b},\nabla\rangle$, see Hypothesis \ref{hyp_base}(iii)) is summable with respect to all the measures $|\mu_j^i|$. This gives a first partial information on the structure of the spaces $L^p(\Rd,|\mu_j^i|)$ and, combined with Theorem \ref{parodontax}(i), shows that $|\T(t)\f|$ is in $L^p(\Rd,|\mu_j^i|)$ for any $\f\in B_b(\Rd,\Rm)$ and $p\le p_*$ for a suitable exponent $p_*$ (explicitly computable). Then, in Theorem \ref{chimica} we characterize all the systems of invariant measures $\{\mu_j: j=1,\ldots,m\}$ such that the above function $\zeta$ belongs to $L^1(\Rd,\mu_j)$ for any $j$: they are linear combinations of the measures $\mu_j^i$ in the sense that there exist real constants $c_1,\ldots,c_m$ such that $\mu_j=\sum_{i=1}^mc_i\mu_j^i$ for any $j=1,\ldots,m$. This result shows, in particular, that systems of invariant measures are infinitely many.
Among all the systems of invariant measures, the $m$ systems $\{\mu_j^i: j=1,\ldots,m\}$ have a very relevant peculiarity: as the invariant measure of the scalar case, they are related to the long-time behaviour of the function $\T(t)\f$. More precisely, in Theorem \ref{venditti} we show that
$(\T(t)\f)_i$ converges to $\sum_{j=1}^m\int_{\Rd}f_jd\mu^i_j$ for any $i=1,\ldots,m$ and any
bounded and continuous function $\f$, locally uniformly in $\Rd$.

Finally, we confine to a particular case where the invariant measure of the scalar operator ${\mathcal A}$ is explicit and provide a sufficient condition for a system
of measures, absolutely continuous with respect to the Lebesgue measure, to be a system of invariant measures for $\T(t)$. Based on this result, we provide some concrete examples of systems of invariant measures, which all consist of signed measures.

The paper is organized as follows. In Section \ref{sect-2} we first introduce the main assumptions on the coefficients of the
operator $\A$ that we use in the paper. In particular, these assumptions guarantee the existence of both the semigroups $T(t)$ and $\T(t)$
and the invariant measure $\mu$ of the semigroup $T(t)$. We also provide a class of elliptic systems which satisfy our assumptions. Then, we prove that the so-called {\it weak generator} can be applied to the semigroup $\T(t)$
and characterize its domain. Subsection \ref{subsect-2.3} is devoted to pointwise gradient estimates, which relate the jacobian matrix of $\T(t)\f$ to the scalar semigroup $T(t)$
applied to the jacobian matrix of $\f$ or to $\f$ itself. This kind of estimates have been already proved in \cite{ALP} when the semigroup $\T(t)$ is associated to an elliptic operator with a nontrivial potential term. In that case the presence of the potential term was crucial to deduce the estimates.
To conclude, in Subsection \ref{subsect-2.4}, we prove some further relevant properties of the semigroup $\T(t)$ that we
need in the paper. In Section \ref{sect-3}, the main body of the paper, we prove the existence of systems of invariant measures for the semigroup $\T(t)$, study their main properties as well as the asymptotic behaviour of the function $\T(t)\f$ when $t$ tends to $+\infty$ and $\f$ is bounded and continuous over $\Rd$. Finally, in Appendix \ref{sect-app} we collect elliptic and parabolic {\it a priori} estimates which we use in the paper.

\subsection*{Notation.}
Vector-valued functions are displayed in bold style. Given a function $\f$ (resp. a sequence
$(\f_n)$) as above, we denote by $f_i$ (resp. $f_{n,i}$) its $i$-th component (resp. the $i$-th component
of the function $\f_n$).
By $B_b(\Rd;\Rm)$ we denote the set of all the bounded Borel measurable functions $\f:\Rd\to\Rm$, where $\|\f\|_{\infty}^2=\sum_{k=1}^m\sup_{x\in\Rd}|f_k(x)|^2$.
For any $k\ge 0$, $C^k_b(\R^d;\R^m)$ is the space of all $\f:\Rd\to\Rm$
whose components belong to $C^k_b(\R^d)$, where ``$b$'' stays for bounded. Similarly, the subscripts ``$c$'', ''$0$'' and ``{\it loc}'' stands for {\it compactly supported}, {\it vanishing at infinity} and {\it locally}, respectively.
The symbols $D_tf$, $D_i f$ and $D_{ij}f$, respectively, denote the time derivative, the first-order spatial derivative with respect to the $i$-th variable and
 the second-order spatial derivative with respect to the $i$-th and $j$-th variables.
We write $J_x\uu$ for the Jacobian matrix of $\uu$ with respect to the spatial variables, omitting the subscript $x$ when no confusion may arise.
By $\boldsymbol{e}_j$ we denote the $j$-th vector of the Euclidean basis of $\R^m$.
Finally, throughout the paper we denote by $c$ a positive constant, which may vary from line to line and, if not otherwise specified, may depend at most on $d$, $m$.
We write $c_{\delta}$ when we want to stress that the constant depends on $\delta$.

\section{Assumptions, preliminary results and gradient estimates}
\label{sect-2}
\subsection{Assumptions and preliminary results}
We consider the following standing assumptions on the coefficients of the operator $\A$ defined
in \eqref{finale} which we split into the sum of two differential operators as follows:
\begin{eqnarray*}
(\A\boldsymbol{\zeta})_j(x)=:({\mathcal A}\zeta_j)(x)+\sum_{k=1}^d\sum_{i=1}^m(B_k(x))_{ji}D_k\zeta_i(x)
\end{eqnarray*}
for any $x\in\Rd$ and any smooth enough function $\bm\zeta:\Rd\to\Rm$.

\begin{hyp}\label{hyp_base}
\begin{enumerate}[\rm(i)]
\item
The coefficients $q_{ij}=q_{ji}$ belong to
$C^{1+\alpha}_{\rm loc}(\R^d)$, for any $i,j=1,\ldots,d$ and some $\alpha\in (0,1)$.
Moreover, $\lambda_0:=\inf_{x\in\Rd}\lambda_Q(x)>0$ where $\lambda_Q(x)$ denotes the minimum eigenvalue of the matrix $Q(x)$;
\item
the coefficients $b_i$ and the entries of the matrices $B_i$ $(i=1,\ldots,d)$ belong to $C^{1+\alpha}_{\rm{loc}}(\R^d)$. Moreover, there exists a nonnegative function $\psi$ such that $|(B_{i})_{jk}|\leq \psi$ in $\Rd$, for any $j,k=1,\ldots,m$, $i=1,\ldots,d$ and $\xi:= \sup_{\Rd}\lambda_Q^{-\frac{1}{2}}\psi<+\infty$;
\item
there exist $a_*\in \R$, $c_*>0$ and a (Lyapunov) function
$1\le\varphi\in C^2(\R^d)$, blowing up as $|x|$ tends to $+\infty$, such that
${\mathcal A}\varphi\leq a_*-c_*\varphi$,
where ${\mathcal A}={\rm Tr}(QD^2)+\langle {\bf b}, \nabla\rangle$, and ${\bf b}=(b_1,\ldots,b_d)$;
\item
there exists $p_0\in (1,2]$ such that
\begin{equation*}
0  >\sigma_{p_0}:=\sup_{x\in\Rd}\bigg[\Lambda_{J_x{\bf b}}(x)\!+\!\bigg( \sum_{i,j=1}^d|D_jB_i(x)|^2\bigg )^{\frac{1}{2}}\!+\!\frac{d(m\xi\!+\!dk(x)(\lambda_Q(x))^{-\frac{1}{2}})^2}{4\min\{1,p_0-1\}}\bigg ],
\end{equation*}
where $\Lambda_{J_x{\bf b}}(x)$ is any function which bounds from above the quadratic form associated to $(J_x{\bf b})(x)$ and $\displaystyle k(x)=\max_{1\le i,j,h\le d}|D_hq_{ij}(x)|$ for any $x\in\Rd$;
\item
there exist two constants $\gamma>2$ and $c_{\gamma}>0$ such that $\psi^{\gamma}\le c_{\gamma}\varphi$ in $\Rd$.
\end{enumerate}
\end{hyp}

\begin{example}
{\rm Let $\A$ be the second order elliptic differential operator defined in \eqref{finale} with
\begin{align*}
Q(x)=(1+|x|^2)^pQ^0(x), \quad b(x)=-b_0x(1+|x|^2)^r, \quad B_i(x)=(1+|x|^2)^{s_i}B_i^0(x),
\end{align*}
for any $x\in\Rd$, where $Q^0(x)$ is a positive definite $d\times d$-matrix for any $x\in\Rd$, $b_0>0$, $B_i^0(x)$ are $m\times m$-matrices for any $x\in\Rd$,
the entries of the matrix-valued functions $Q^0$ and $B_i^0$ ($i=1,\ldots,d$) belong to $C^{1+\alpha}_b(\Rd)$ for some $\alpha\in (0,1)$,
and $p,r,s_i\ge 0$ satisfy the conditions $s_{\max}:=\max\{s_1,\ldots,s_d\}\leq p/2$ and $r>\max\{p,s_{\max}\}$. Clearly, Hypothesis \ref{hyp_base}(i) and (ii) are satisfied. In particular, $\lambda_Q(x)\ge\lambda_1(1+|x|^2)^p$
for any $x\in\Rd$ and $\xi\le\xi_0:=\sqrt{\lambda_1}\max_{i=1,\ldots,d}\|B_i^0\|_{\infty}$, where $\lambda_1$ denotes the infimum over $\Rd$ of the minimum eigenvalue of the matrix $Q^0(x)$.
Moreover, for any choice of $h\in\N$ the function $\varphi(x)=(1+|x|^2)^h$ is a Lyapunov function for the operator $\mathcal A={\rm Tr}(QD^2)+\langle {\bf b},\nabla \rangle$, so that Hypotheses \ref{hyp_base}(iii) and (v) are trivially satisfied, this latter one for any choice of $\gamma>2$. Finally,  since $k(x)\le c_0(1+|x|^2)^p$ for any $x\in\Rd$ and some positive constant $c_0$, we obtain that, if there exists $p_0\in(1,2]$ such that
\begin{align}
\label{denuncia}
\bigg (\sum_{i,j=1}^d(2s_i\|B_i^0\|_{\infty}+\|D_jB_i^0\|_{\infty})^2\bigg )^{\frac{1}{2}}+\frac{d(m\xi_0+dc_0\lambda_1^{-\frac{1}{2}})^2}{4\min\{1,p_0-1\}}<b_0,
\end{align}
then Hypothesis \ref{hyp_base}(iv) is fulfilled. Indeed,
\begin{align*}
& \Lambda_{J_x{\bf b}}(x)\!+\!\bigg( \sum_{i,j=1}^d|D_jB_i^0(x)|^2\bigg )^{\frac{1}{2}}\!+\!\frac{d(m\xi\!+\!dk(x)(\lambda_Q(x))^{-\frac{1}{2}})^2}{4\min\{1,p_0-1\}} \notag \\
\le &(1+|x|^2)^r\bigg[(1+|x|^2)^{s_{\max}-r}\bigg (\sum_{i,j=1}^d(2s_i(1+|x|^2)^{-\frac{1}{2}}\|B_i^0\|_{\infty}+\|D_jB_i^0\|_{\infty})^2\bigg )^{\frac{1}{2}}-b_0\notag \\
&\phantom{(1+|x|^2)^r\bigg [\,}+\frac{d(m\xi_0+dc_0\lambda_1^{-\frac{1}{2}}(1+|x|^2)^{-\frac{p}{2}})^2}{4\min\{1,p_0-1\}}(1+|x|^2)^{p-r}\bigg]
\end{align*}
for any $x\in\Rd$. Due to our choice of the parameters $p$, $r$ and $s$, the function in square brackets assume its maximum value at $x=0$, which is negative when \eqref{denuncia} is satisfied.}
\end{example}

Under (a part of) Hypotheses \ref{hyp_base} the following result holds true.

\begin{thm}[Section 4, \cite{metafune-pallara-wacker1}]\label{maremaremare}
Assume that Hypotheses $\ref{hyp_base}(i)$-$(iii)$ are satisfied. Then, for any $f\in C_b(\Rd)$ the Cauchy problem
\begin{align*}
\left\{
\begin{array}{lll}
D_tu(t,x)={\mathcal A}u(t,x), &\quad t\in(0,+\infty), & x\in\Rd, \\[1mm]
u(0,x)=f(x), & \quad & x\in\Rd,
\end{array}
\right.
\end{align*}
admits a unique classical solution $u_f\in C_b([0,+\infty)\times\Rd)\cap C^{1+\alpha/2,2+\alpha}_{\rm loc}((0,+\infty)\times\Rd)$ satisfying $\|u_f(t,\cdot)\|_\infty\leq \|f\|_\infty$ for any $t\geq0$. Moreover, if we set $T(t)f:=u_f(t,\cdot)$ for any $t>0$ and $f\in C_b(\Rd)$, and $T(0)=Id$, then $T(t)g_n$ converges to $T(t)g$ locally uniformly as $n$ tends to $+\infty$, for any $t>0$ and any
bounded sequence $(g_n)\subset C_b(\Rd)$ converging pointwise in $\Rd$ to $g\in C_b(\Rd)$.
\end{thm}

The semigroup $T(t)$ admits the following integral representation:
\begin{equation}
(T(t)f)(x)=\int_{\Rd}f(y)p(t,x,dy),\qquad\;\,t>0,\;\,x\in\Rd,\;\,f\in C_b(\Rd),
\label{valigia}
\end{equation}
where the $p(t,x,dy)$'s are probability Borel measures, each of them equivalent to the Lebesgue measure. In addition, Hypothesis \ref{hyp_base}$(iii)$ implies that $T(t)$ admits a unique invariant measure $\mu$, that is a
Borel probability measure which satisfies the condition
\begin{align*}
\int_{\Rd}T(t)fd\mu=\int_{\Rd}fd\mu, \quad t\geq0,
\end{align*}
for any $f\in C_b(\Rd)$. This result is due to Has'minskii (see \cite[Chapter 3, Theorem 5.1]{HAS}).
Hypothesis \ref{hyp_base}(iii) and the result in \cite{kunze-lorenzi-lunardi}, (see Lemma 5.3), imply that $\varphi \in L^1(\Rd,\mu)$, that one can apply the scalar semigroup $T(t)$ to the function $\varphi$ and
\begin{equation}
(T(t)\varphi)(x)\le c^{-1}_*a_*+\varphi(x),\qquad\;\,t>0,\;\,x\in\Rd,
\label{addio-00}
\end{equation}
where the constants $a_*$ and $c_*$ are the same as in the quoted hypothesis.

\medskip

Hypotheses \ref{hyp_base}(i)-(iii) ensure that also the vector-valued Cauchy problem
associated to the operator $\A$ admits a unique classical solution.

\begin{thm}[Theorem 2.9, \cite{AALT}]
\label{point_prop}
Let Hypotheses $\ref{hyp_base}$ be satisfied. Then, for any $\f\in C_b(\Rd;\R^m)$ the Cauchy problem
\begin{align}
\left\{
\begin{array}{lll}
D_t\uu(t,x)=\A\uu(t,x),\quad & t\in(0,+\infty), & x\in\Rd,\\[1mm]
\uu(0,x)=\f(x), & & x\in\Rd,
\end{array}
\right.
\label{alimentari}
\end{align}
admits a unique locally in time bounded classical solution
$\uu$. Moreover, $\uu$ belongs to $C^{1+\alpha/2,2+\alpha}_{\rm loc}((0,+\infty)\times\Rd;\R^m)$ and
\begin{equation}
\label{pointwise}
|\uu(t,x)|^2\leq e^{2\beta t}(T(t)|\f|^2)(x),\qquad\;\,(t,x)\in [0,+\infty)\times \Rd,
\end{equation}
where $\beta=4^{-1}m^2d\xi^2$ and $\xi$ is defined in Hypothesis $\ref{hyp_base}(ii)$. As a byproduct,
\begin{equation}
|\uu(t,x)|\leq e^{\beta t}\|\f\|_{\infty},\qquad\;\,(t,x)\in [0,+\infty)\times \Rd,\;\,\f\in C_b(\Rd;\R^m).
\label{rettore}
\end{equation}
\end{thm}

\begin{rmk}
\label{rem-2.1}
{\rm In the scalar case, the semigroup associated to an elliptic operator ${\mathcal A}$ with zero potential term is always contractive as
a straightforward consequence of the maximum principle. In the vector-valued case, the maximum principle does not hold if
the elliptic operator is coupled at the first-order as our operator is. This shows why we cannot expect
\eqref{pointwise} with $\beta=0$. On the other hand, we can expect neither the semigroup $\T(t)$ to be bounded. Indeed, consider the
two-dimensional elliptic operator $\boldsymbol{\mathcal A}$ defined on smooth functions $\boldsymbol{\zeta}=(\zeta_1,\zeta_2)$ by
\begin{eqnarray*}
\boldsymbol{\mathcal A}\boldsymbol{\zeta}=(D_{xx}\zeta_1-D_x\zeta_1+D_x\zeta_2,D_{xx}\zeta_2-5D_x\zeta_1+D_x\zeta_2).
\end{eqnarray*}
A straightforward computation shows that, if $\f(x)=(\cos(x),2\sin(x)+\cos(x))$ for any $x\in\R$, then
$(\T(t)\f)(x)=(e^t\cos(x),e^t(2\sin(x)+\cos(x))$ for any $t>0$ and $x\in\R$. As a consequence, $\|\T(t)\f\|_{\infty}\ge e^t$ for any $t>0$.}
\end{rmk}

Theorem \ref{point_prop} allows to introduce the vector-valued semigroup
$\T(t)$ of bounded linear operators in $C_b(\Rd;\Rm)$ by setting $(\T(t)\f)(x):=\uu(t,x)$,
for any $t\ge 0$ and $x\in\Rd$, where $\uu$ is the classical solution to the Cauchy problem \eqref{alimentari}. By \cite[Theorem 3.3]{AALT},
the following integral representation formula
\begin{align}
({\bf T}(t)\f)_i(x)=\sum_{j=1}^m\int_{\Rd}f_j(y)p_{ij}(t,x,dy),\qquad\;\,\f\in C_b(\Rd;\R^m),\;\,i=1,\ldots,m,
\label{domanda}
\end{align}
holds true for any $t>0$, where $\{p_{ij}(t,x,dy): t>0,\, x\in\Rd,\, i,j=1,\ldots, m\}$ is a family
of finite signed Borel measures, which are absolutely continuous with respect to the Lebesgue measure. Formula \eqref{domanda} allows to extend easily the semigroup $\T(t)$ to
$B_b(\Rd;\Rm)$, by approximating any $\f\in B_b(\Rd;\Rm)$ by a bounded sequence $(\f_n)\subset C_b(\Rd;\Rm)$, which converges to $\f$ almost everywhere (with respect to each measure $p_{ij}(t,x,dy)$) in $\Rd$. We can also use this formula to characterize the boundedness of the semigroup $\T(t)$ in terms of the boundedness of the total variations $|p_{ij}|(t,x,dy)$ of the measures
$p_{ij}(t,x,dy)$ ($t>0, x\in\Rd)$.

\begin{prop}
\label{lem-2.1}
The semigroup $\T(t)$ is bounded in $C_b(\Rd;\Rm)$ if and only if
the family of measures $\{|p_{ij}|(t,x,dy): t>0, x\in\Rd\}$ is bounded for any $i,j=1,\ldots,m$.
\end{prop}

\begin{proof}
Suppose that the semigroup $\T(t)$ is bounded and fix $t>0$, $x\in\Rd$ and $i,j\in\{1,\ldots,m\}$.
By the Hahn decomposition theorem, there exist two disjoint Borel sets $P_x^{ij}$ and $N_x^{ij}$, whose union is $\Rd$, which are, respectively,
the supports of the positive part $p_{ij}^+(t,x,dy)$ and of the negative part $p_{ij}^-(t,x,dy)$ of the measure $p_{ij}(t,x,dy)$.

Since each measure $p_{ij}(t,x,dy)$ is absolutely continuous with respect to the Lebesgue measure, we can determine two bounded sequences $(f_n),(g_n)\subset C_b(\Rd)$
converging to $\chi_{P^{ij}_x}$ and
$\chi_{N^{ij}_x}$, respectively, almost everywhere in $\Rd$ (with respect to the measure $p_{ij}(t,x,dy)$). We set $\f_n=f_n \e_j$ and
$\g_n=g_n \e_j$ and observe that
\begin{align*}
p_{ij}^+(t,x,\Rd)& =p_{ij}(t,x,P^{ij}_x)=\int_{\Rd}\chi_{P^{ij}_x}(y)p_{ij}(t,x,dy) \\
& =\lim_{n\to +\infty}\int_{\Rd}f_n(y)p_{ij}(t,x,dy)=\lim_{n\to +\infty}(\T(t)\f_n)_i(x).
\end{align*}
From the boundedness of each operator $\T(t)$ and the previous formula we deduce that $\sup_{(t,x)\in (0,+\infty)\times\Rd}p_{ij}^+(t,x,\Rd)<+\infty$.
Replacing $(\f_n)$ with the sequence $(\g_n)$ and arguing similarly, we conclude that $\sup_{(t,x)\in (0,+\infty)\times\Rd}p_{ij}^-(t,x,\Rd)<+\infty$
Thus, the family $\{|p_{ij}|(t,x,dy): t>0, x\in\Rd\}$ is bounded for any $i,j=1,\ldots,m$.

The other part of the statement follows trivially from the representation formula \eqref{domanda}. Indeed, let $M>0$ be any constant such that $|p_{ij}|(t,x,\Rd)\le M$ for any $t>0$, $x\in\Rd$, $i,j=1,\ldots,m$. Then, we can estimate
\begin{align*}
|(\T(t)\f)_i(x)|\le &\bigg |\sum_{j=1}^m\int_{\Rd}f_j(y)p_{ij}(t,x,dy)\bigg |
\le \sum_{j=1}^m\int_{\Rd}|f_j(y)||p_{ij}|(t,x,dy)\\
\le &M\sum_{j=1}^m\|f_j\|_{\infty}
\le M\sqrt{m}\|\f\|_{\infty}
\end{align*}
for any $t>0$, $x\in\Rd$, $i=1,\ldots,m$, $\f\in C_b(\Rd;\Rm)$  and we are done.
\end{proof}

\subsection{The weak generator of $\T(t)$}
As it is known, (in general) semigroups associated with scalar elliptic operators with unbounded coefficients are not strongly continuous in $C_b(\Rd)$. However,
it is possible to associate the so called \emph{weak} generator to them, in two different ways (see, e.g., \cite{newbook, metafune-pallara-wacker1}).
This approach has been already extended to vector-valued weakly coupled elliptic operators with unbounded coefficients in \cite{DelLor11OnA}. We show that it works also in our case.

The first approach considered leads to the definition of the resolvent of the generator by means of the Laplace transform. Indeed, thanks to estimate \eqref{rettore},
the function $t\mapsto e^{-\lambda t}(\T(t)\f)(x)$ is continuous and belongs to $L^1((0,+\infty))$ for any $\lambda>\beta$ and $x \in \Rd$. Hence, we can define bounded operators $\Res(\lambda)$ in $C_b(\Rd;\R^m)$ for $\lambda>\beta$ through the formula
\begin{equation}
(\Res(\lambda)\f)(x)=\int_0^{+\infty}e^{-\lambda t}(\T(t)\f)(x)dt,\qquad\;\, x \in \R^d,\,\f\in C_b(\Rd;\R^m).
\label{Res}
\end{equation}
The family $\{\Res(\lambda): \lambda >\beta\}$ satisfies the resolvent identity and every operator $\Res(\lambda)$ is injective in $C_b(\Rd;\R^m)$,
so that there exists a unique closed operator $(\widehat{\boldsymbol{A}}, \widehat{\boldsymbol{D}})$ such that $\Res(\lambda)=(\lambda-\widehat{\boldsymbol{A}})^{-1}$
for any $\lambda>\beta$, i.e., $\lambda-\widehat{\boldsymbol{A}}$ is bijective from $\widehat{\boldsymbol{D}}$ onto $C_b(\Rd;\R^m)$ for any $\lambda>\beta$ (see e.g., 
\cite[Section 8.4]{yosida}).

On the other hand, we can define the weak generator of $\T(t)$ mimicking the classical definition of infinitesimal generator of a strongly continuous semigroup, by setting
\begin{align*}
&\widetilde{\boldsymbol{D}}=\Bigg\{\uu\in C_b(\Rd;\R^m):\,\, \sup_{t>0}\bigg\|\frac{\T(t)\uu-\uu}{t}\bigg\|_{\infty}<+\infty\; {\rm and}\; \exists \g\in C_b(\Rd;\R^m)\\
&\phantom{\widetilde{\boldsymbol{D}}=\Bigg\{\;\,}\textrm{ such that } \lim_{t \to 0^+}\frac{(\T(t)\uu)(x)-\uu(x)}{t}=\g(x) \,\,\forall x \in \Rd\Bigg\},\\[2mm]
&\widetilde{\boldsymbol{A}}\uu=\g,\quad \uu \in {\widetilde{\boldsymbol{D}}}.
\end{align*}

The same arguments used in the scalar case (see \cite[Section 5]{metafune-pallara-wacker1}) show that $\hat{\boldsymbol{A}}$ and $\widetilde{\boldsymbol{A}}$ actually coincide.
Thus, we can set $\boldsymbol{A}:=\hat{\boldsymbol{A}}=\widetilde{\boldsymbol{A}}$, $\boldsymbol{D}:=\hat{\boldsymbol{D}}=\widetilde{\boldsymbol{D}}$ and prove
the following characterization for the weak generator $(\boldsymbol A, \boldsymbol D)$.

\begin{prop}
\label{incarico}
The weak generator $(\boldsymbol A, \boldsymbol D)$ of the semigroup $\T(t)$ coincides with the operator $(\A, D_{\rm max}(\A))$, where
$D_{\rm max}(\A)$ denotes the domain of the maximal realization of the operator $\A$ in $C_b(\Rd;\R^m)$, i.e.,
\begin{eqnarray*}
D_{\rm max}(\A)=\bigg\{\uu \in C_b(\Rd;\R^m)\cap\bigcap_{1\le p<\infty}W^{2,p}_{\rm loc}(\Rd,\R^m): \A\uu \in C_b(\Rd;\R^m)\bigg\}.
\end{eqnarray*}
Moreover, for any $\f \in D_{\rm max}(\A)$ and $t>0$ the function $\T(t)\f$
belongs to $D_{\rm max}(\A)$ and $\T(t)\A\f=\A\T(t)\f$ for any $t>0$. Finally, for any $\f\in D_{\rm max}(\A)$ and $x\in\Rd$, the function $\T(\cdot)\f$
is continuously differentiable in $[0,+\infty)$ with
\begin{align}
\label{fantine}
\frac{d}{dt}(\T(t)\f)(x)=(\T(t)\A\f)(x), \qquad t\geq0.
\end{align}
\end{prop}
\begin{proof}
Fix $\uu \in \boldsymbol D$, $\lambda>\beta$ and let $\f \in C_b(\Rd;\R^m)$ be such that $\uu= \Res(\lambda, \boldsymbol A)\f$. For any $n\in \N$ let the function $\uu_n$
be defined by
\begin{equation*}
\uu_n(x)=\int_{1/n}^n e^{-\lambda t}(\T(t)\f)(x)dt,\qquad\;\, x \in \Rd.
\end{equation*}
Taking estimate \eqref{rettore} into account we deduce that $\sup_{n\in \N}\|\uu_n\|_\infty<+\infty$ and $\uu_n$ converges to $\uu$ uniformly in $\Rd$ as $n$ tends to
$+\infty$. Moreover,
\begin{align}\label{aun}
\A \uu_n(x)&= \int_{1/n}^n e^{-\lambda t}(\A\T(t)\f)(x)dt= \int_{1/n}^n e^{-\lambda t}(D_t\T(t)\f)(x)dt\notag\\
&= e^{-\lambda n}(\T(n)\f)(x)-e^{-\lambda/n}(\T(n^{-1})\f)(x)+\lambda \uu_n,
\end{align}
whence $\sup_{n\in \N}\|\A \uu_n\|_\infty<+\infty$ as well. Estimate \eqref{lp_int} yields that
$\|u_{n,k}\|_{W^{2,p}(B(0,r))}$ $\le c_{p,r}$ for any $n\in\N$, $k=1,\ldots, m$, $r>0$ and $p\in(1,+\infty)$.
By compactness, there exist $\bar{u}_1,\ldots, \bar{u}_m\in W^{2,p}_{\rm loc}(\Rd)$ such that $u_{n,k}$ converges to $\bar{u}_k$ ($k=1, \ldots, m$) strongly in $W^{1,p}(B(0,r))$ and weakly in $W^{2,p}(B(0,r))$ for any $r>0$. Thus, we infer that $\bar{u}_k=u_k$, whence $u_k\in W^{2,p}_{\rm loc}(\Rd)$ for any $k=1, \ldots,m$ and $p\in [1,+\infty)$.

Since $\uu_n$ converges to $\uu$ weakly in $W^{2,p}_{\rm{loc}}(\Rd;\Rm)$, $\A\uu_n$ converges to $\A\uu$ weakly in $L^p_{\rm loc}(\Rd;\R^m)$. On the other hand formula  \eqref{aun} implies that $\A\uu_n$ converges to $\lambda \uu-\f$ locally uniformly in $\Rd$ as $n$ tends to $+\infty$. As a byproduct, we conclude that $\A\uu=\lambda \uu-\f\in C_b(\Rd;\R^m)$, i.e., $\A\uu=\boldsymbol A \uu \in C_b(\Rd;\R^m)$. This yields that $\boldsymbol D\subset D_{\rm max}(\A)$ and $\A\uu=\boldsymbol A \uu$ for any $\uu\in \boldsymbol D$.

From the definition of $\boldsymbol D$ and the inclusion $\boldsymbol D\subset D_{\rm max}(\A)$, it follows that $\boldsymbol D= D_{\rm max}(\A)$ if and only if $\lambda-\boldsymbol A$ is injective on $D_{\rm max}(\A)$. To prove the injectivity of $\lambda-\boldsymbol A$ on $D_{\rm max}(\A)$ we show that $\uu=0$ is the unique solution to the equation $\lambda \uu-\A\uu=0$ in $D_{\rm max}(\A)$. To this aim, let $\uu\in D_{\rm max}(\A)$ solve such an equation. A straightforward computation reveals that $|\uu|^2$
\begin{equation}\label{biagio}
\lambda |\uu|^2-\frac{1}{2}{\mathcal A}|\uu|^2= -\sum_{i,j=1}^d\sum_{k=1}^m q_{ij}D_i u_kD_j u_k+ \sum_{i=1}^d\sum_{j,k=1}^m (B_i)_{kj}u_kD_iu_j.
\end{equation}
Using Hypotheses \ref{hyp_base}(i)-(ii) we can estimate the term in the right-hand side of \eqref{biagio} as follows
\begin{align*}
-\sum_{i,j=1}^d\sum_{k=1}^m q_{ij}D_i u_kD_j u_k+ \sum_{i=1}^d\sum_{j,k=1}^m (B_i)_{kj}u_kD_iu_j &\le -\lambda_Q |J_x\uu|^2+m\sqrt{d}\psi|\uu||J_x \uu|\\
&\le (\varepsilon \psi^2-\lambda_Q) |J_x \uu|^2+\frac{m^2d}{4\varepsilon}|\uu|^2\\
& \le \lambda_Q(\varepsilon \xi^2-1)|J_x \uu|^2+\frac{m^2d}{4\varepsilon}|\uu|^2.
\end{align*}
Choosing $\varepsilon=\xi^{-2}$ we obtain that $\lambda |\uu|^2-\hat{\mathcal A}_0|\uu|^2\le 0 $ where $\hat{\mathcal A}_0=\frac{1}{2}{\mathcal A}+\frac{m^2d\xi^2}{4}$. Hypothesis \ref{hyp_base}(iii) and the maximum principle in \cite[Theorem 3.1.6]{newbook} yield that $\uu=0$.

To complete the proof, let us fix $\f\in D_{\rm max}(\A)$ and $t>0$. Estimate \eqref{pointwise} and the semigroup law show that
\begin{align}\label{solesole}
\bigg|\frac{\T(h)-I}{h}\T(t)\f-\T(t)\A\f\bigg|^2&= \bigg|\T(t)\bigg (\frac{\T(h)-I}{h}\f-\A\f\bigg )\bigg|^2\notag\\
&\le e^{2\beta t}T(t)\bigg(\bigg|\frac{\T(h)-I}{h}\f-\A\f\bigg|^2\bigg).
\end{align}
Since $\f\in D_{\rm max}(\A)=\boldsymbol D$ it holds that $\sup_{h\in (0,1)}\|(\T(h)\f-\f)h^{-1}\|_{\infty}<+\infty$ and $(\T(h)\f-\f)h^{-1}$
converges pointwise to $\A\f$ as $h$ tends to $0^+$. Thanks to Theorem \ref{maremaremare}, the right-hand side of
\eqref{solesole} vanishes locally uniformly as $h$ tends to $0^+$.
Thus, $(\T(h)\T(t)\f-\T(t)\f)h^{-1}$ converges to $\T(t)\A\f$ locally uniformly as $h$ tends to $0^+$. Moreover, the semigroup law, estimate \eqref{rettore} and the fact that
$\sup_{h\in (0,1)}\|(\T(h)\f-\f)h^{-1}\|_\infty<+\infty$ imply that $\sup_{h\in (0,1)}\|(\T(h)\T(t)\f-\T(t)\f)h^{-1}\|_\infty<+\infty$;
whence we deduce that $\T(t)\f\in \boldsymbol D=D_{\rm max}(\A)$ and $\A\T(t)\f=\T(t)\A\f$.

To show formula \eqref{fantine}, we fix $x\in\Rd$, $\f\in D_{\rm max}(\A)$ and observe that estimate \eqref{solesole} implies that the right-derivative of the function $(\T(\cdot)\f)(x)$ exists in $[0,+\infty)$ and coincide with the function $(\T(\cdot)\A\f)(x)$. Since this function is continuous in $[0,+\infty)$, formula \eqref{fantine} follows at once.
\end{proof}

\subsection{Pointwise gradient estimates}
\label{subsect-2.3}
In this subsection we provide sufficient conditions which ensure pointwise gradient estimates for the
vector-valued semigroup $\T(t)$. As a by product, under additional assumptions we show
that the function $\T(\cdot)\f$ is uniformly bounded in $[0,+\infty)\times B(0,r)$ for any $\f \in C_b(\Rd;\Rm)$ and $r>0$.

\begin{prop}
\label{premier}
Under Hypotheses $\ref{hyp_base}(i)$-$(iv)$, for any $p \ge p_0$ and any $\f\in C^1_b(\Rd;\Rm)$ it holds that
\begin{align}
|(J_x\T(t)\f)(x)|^p\leq e^{p\sigma_{p_0} t}(T(t)|J\f|^p)(x), \qquad\;\, t>0,\;\,x\in\Rd,
\label{capitanamerica}
\end{align}
where $\sigma_{p_0}$ is defined in Hypothesis $\ref{hyp_base}(iv)$.
\end{prop}

\begin{proof}
Let $\f$ and $p$ be as in the statement. We claim that
\begin{align}
\label{muratori}
|J_x\uu_n(t,x)|^p\leq e^{p\sigma_{p_0} t}(T_n^{\mathcal N}(t)|J \f|^p)(x), \qquad\;\,t>0,\;\,x\in B(0,n),
\end{align}
where $\uu_n$ is the unique classical solution to the homogeneous Neumann-Cauchy problem associated to the equation $D_t\uu=\A\uu$ in $B(0,n)$
and $T_n^{\mathcal N}(t)$ denotes the semigroup associated to the realization of the operator $\mathcal{A}$ in $C_b(\overline{B(0,n)})$ with homogeneous Neumann boundary conditions.
Once \eqref{muratori} is proved, \eqref{capitanamerica} will follow simply observing that $\uu_n$, converges to $\T(\cdot)\f$ in $C^{1,2}(K)$ for any compact set $K\subset (0,+\infty)\times \Rd$ (see \cite[Remark 2.10]{AALT}).

So, let us prove \eqref{muratori}. Fix $\varepsilon>0$ and for any $n\in\N$ set $v_n:=(|J_x\uu_n|^2+\varepsilon_p)^{p/2}$,
where $\varepsilon_p=\varepsilon>0$ if $p\in (p_0,2)$ and $p_0<2$ and $\varepsilon_p=0$ otherwise (i.e., $p\ge 2$).
By classical results, $v_n$ belongs to $C^{1,2}((0,T)\times B(0,n))\cap C_b([0,T]\times\overline{B(0,n)})$ for any $T>0$ and
\begin{align*}
\left\{
\begin{array}{ll}
\displaystyle D_tv_n={\mathcal A}v_n+pv^{1-\frac{2}{p}}_n(\psi_1+\psi_2)+pv_n^{1-\frac{4}{p}}\psi_3, \quad & {\rm in}~(0,+\infty)\times B(0,n), \vspace{1mm}\\
\displaystyle\frac{\partial v_n}{\partial \nu}(t,x)\le 0, & {\rm in}~(0,+\infty)\times\partial B(0,n),  \vspace{1mm} \\
v_n(0,\cdot)=(|J\f|^2+\varepsilon_p)^{\frac{p}{2}}, & {\rm in}~B(0,n),
\end{array}
\right.
\end{align*}
where
\begin{align*}
\psi_1 =& \sum_{k=1}^m\langle (J_x{\bf b})\nabla_xu_{n,k},\nabla_xu_{n,k}\rangle
-\sum_{i=1}^d\sum_{k=1}^m|\sqrt{Q}\nabla_xD_iu_{n,k}|^2,\\[2mm]
\psi_2 =& \sum_{i,j=1}^d\langle D_iB_j D_j\uu_n,D_i\uu_n\rangle\!+\!\sum_{i,j=1}^d\langle B_jD_{ij}\uu_n,D_i\uu_n\rangle
\!+\!\sum_{i,j,h=1}^dD_hq_{ij}\langle D_{ij}\uu_n,D_h\uu_n\rangle,\\[2mm]
\psi_3=&(2-p)\sum_{k=1}^m|\sqrt{Q}D^2_xu_{n,k}\nabla_xu_{n,k}|^2.
\end{align*}

We can estimate $\psi_1$ and $\psi_2$ as follows:
\begin{align}\label{stima1}
\psi_1\leq &\Lambda_{Jb}|J_x\uu_n|^2-\sum_{i=1}^d\sum_{k=1}^m|\sqrt{Q}\nabla_xD_iu_{n,k}|^2\nonumber\\
\le &\Lambda_{Jb}|J_x\uu_n|^2-\lambda_Q|D^2_x\uu_n|^2
\end{align}
and
\begin{align}\label{stima2}
|\psi_2| \leq & \bigg (\sum_{i,j=1}^d|D_jB_i|^2\bigg )^{\frac{1}{2}}|J_x\uu_n|^2
+(\sqrt{d}m\xi\sqrt{\lambda_Q}+d^{\frac{3}{2}}k)|J_x\uu_n||D^2_x\uu_n|\nonumber\\
\le & \bigg [\bigg (\sum_{i,j=1}^d|D_jB_i|^2\bigg )^{\frac{1}{2}}+\frac{1}{4a}(\sqrt{d}m\xi+d^{\frac{3}{2}}k\lambda_Q^{-\frac{1}{2}})^2\bigg ]|J_x\uu_n|^2
+a\lambda_Q|D_x^2\uu_n|^2
\end{align}
for any $a>0$. Note that if $p\ge 2$ then $\psi_3\le 0$. In this case, the second part of estimate \eqref{stima1} and estimate \eqref{stima2} with $a=1$ yield that $D_t v_n-\mathcal{A}v_n\le p\sigma_{p_0}v_n$. Otherwise if $p_0<2$ and $p\in (p_0,2)$ then the function $\psi_3$ is nonnegative and can be estimated as follows:
\begin{align}
(2-p)^{-1}\psi_3\le & \sum_{i,j=1}^d\sum_{k=1}^m|\sqrt{Q}\nabla_x D_iu_{n,k}||\sqrt{Q}\nabla_x D_ju_{n,k}||D_iu_{n,k}||D_ju_{n,k}|\notag\\
= & \sum_{k=1}^m\bigg (\sum_{i=1}^d|\sqrt{Q}\nabla_x D_iu_{n,k}||D_iu_{n,k}|\bigg )^2\notag\\
\le & \sum_{k=1}^m\sum_{i=1}^d|\sqrt{Q}\nabla_x D_iu_{n,k}|^2|\nabla_xu_{n,k}|^2\notag\\
\le & |J_x\uu_n|^2\sum_{k=1}^m\sum_{i=1}^d|\sqrt{Q}\nabla_x D_iu_{n,k}|^2.
\label{star-1}
\end{align}
Summing up, from (the first part of) \eqref{stima1}, \eqref{stima2} and \eqref{star-1}, we obtain that
\begin{align*}
D_tv_n-\mathcal Av_n \le &  pv_n^{1-\frac{2}{p}}|J_x\uu_n|^2\bigg [\Lambda_{Jb}+\bigg (\sum_{i,j=1}^d|D_jB_i|^2\bigg )^{\frac{1}{2}}+\frac{1}{4a}(\sqrt{d}m\xi+d^{\frac{3}{2}}k\lambda_Q^{-\frac{1}{2}})^2\bigg ]\\
&+p(1-p)v_n^{1-\frac{2}{p}}\sum_{k=1}^m\sum_{i=1}^d\langle Q\nabla D_iu_{n,k}, \nabla D_iu_{n,k}\rangle+
apv_n^{1-\frac{2}{p}}\lambda_Q|D^2_x\uu_n|^2\\
\le &  pv_n^{1-\frac{2}{p}}|J_x\uu_n|^2\bigg [\Lambda_{Jb}+\bigg (\sum_{i,j=1}^d|D_jB_i|^2\bigg )^{\frac{1}{2}}+\frac{1}{4a}(\sqrt{d}m\xi+d^{\frac{3}{2}}k\lambda_Q^{-\frac{1}{2}})^2\bigg ]\\
&+pv_n^{1-\frac{2}{p}}[(1-p+a)\lambda_Q|D^2_x\uu_n|^2.
\end{align*}
Thus, choosing $a=p-1$, the coefficient in front of $|D^2_x\uu_n|^2$ vanishes and the estimate becomes
$D_tv_n-\mathcal A v_n\leq p{\sigma_{p_0}}v_n-p\sigma_{p_0}\varepsilon_{p_0} v_n^{1-2/p}\le p\sigma_{p_0}v_n-p\sigma_{p_0}\varepsilon_p^{p/2}$.

Now, the procedure is the same in the two cases considered: we set $w_n(t,\cdot):=v_n(t,\cdot)-\varepsilon_p^{p/2}-e^{p\sigma_{p_0}t}T_n^{\mathcal N}(t)((|\nabla \f|^2+\varepsilon_p)^{p/2})$ for any $t\ge 0$
and observe that
\begin{align*}
\left\{
\begin{array}{ll}
\displaystyle D_tw_n-(\mathcal A+p\sigma_{p_0}) w_n\leq 0,\quad & {\rm in}~(0,+\infty)\times B(0,n), \\[1mm]
\displaystyle\frac{\partial w_n}{\partial \nu}\le 0, & {\rm in}~(0,+\infty)\times\partial B(0,n),\\[1mm]
w_n(0,\cdot)=-\varepsilon_p^{\frac{p}{2}}, & {\rm in}~B(0,n).
\end{array}
\right.
\end{align*}
The classical maximum principle yields that $w_n\le 0$ in $(0,+\infty)\times B(0,n)$, i.e.,
$v_n(t,\cdot)\le e^{p\sigma_{p_0}t}T_n^{\mathcal N}(t)((|\nabla \f|^2+\varepsilon_p)^{p/2})+\varepsilon_p^{p/2}$ for any $t>0$ and this concludes the proof if $p\ge 2$. Otherwise, we let $\varepsilon_p$ tend to $0^+$ and again we conclude the proof.
\end{proof}

\begin{prop}
\label{bundes}
Under Hypotheses $\ref{hyp_base}(i)$-$(iv)$, for any $p \ge p_0$ there exists a positive constant $C_p$ such that
\begin{align}
|(J_x\T(t)\f)(x)|^p\leq C_pe^{p\sigma_{p_0}t}(1\vee t^{-\frac{p}{2}})(T(t)|\f|^p)(x), \qquad\;\, t>0,\;\, x\in\Rd,
\label{capitanfuturo}
\end{align}
for any $\f \in C_b(\Rd;\Rm)$.
\end{prop}

\begin{proof}
Here, we take advantage of the notation and the results in the proof of Proposition \ref{premier}.
We actually reduce ourselves to proving that for any $p\ge p_0$ there exists a positive constant $k_p$ such that
\begin{align}
|J_x\uu_n(t,x)|^p\leq k_pt^{-\frac{p}{2}}(T_n^{\mathcal N}(t)|\f|^p)(x), \qquad\;\,t\in (0,1],\;\,x\in B(0,n),
\label{monomi}
\end{align}
for any $n\in\N$ and $\f\in C_b(\Rd;\Rm)$.
Once \eqref{monomi} is proved, letting $n$ tend to $+\infty$ we obtain
\begin{align}
|(J_x\T(t)\f)(x)|^p\leq k_pt^{-\frac{p}{2}}(T(t)|\f|^p)(x), \qquad\;\,t\in (0,1],\;\,x\in B(0,n).
\label{monomi-ii}
\end{align}

Finally, estimate \eqref{capitanfuturo} will follow using the semigroup rule, estimates \eqref{capitanamerica}, \eqref{monomi-ii} and the positivity of the scalar semigroup $T(t)$.
Indeed, if $t>1$ then
\begin{align*}
|(J_x\T(t)\f)(x)|^p=&|(J_x\T(t-1)\T(1)\f)(x)|^p\le e^{p\sigma_{p_0}(t-1)}(T(t-1)|J_x\T(1)\f|^p)(x)\\
\le &k_pe^{p\sigma_{p_0}(t-1)}(T(t)|\f|^p)(x)
\end{align*}
for any $x\in\Rd$, and \eqref{capitanfuturo} follows with $C_p=k_pe^{p|\sigma_{p_0}|}$.

So, let us prove estimate \eqref{monomi}. First, we set
\begin{eqnarray*}
v_n(t,x):=(|\uu_n(t,x)|^2+\gamma t|J_x\uu_n(t,x)|^2+\varepsilon_p)^{\frac{p}{2}}
\end{eqnarray*}
for any $t\in (0,1]$, $x\in \Rd$ and $n \in \N$, where $\varepsilon_p$ is as in the proof of Proposition \ref{premier} and $\gamma$
is a positive constant which will be fixed later.
For any $n \in \N$, the function $v_n$ belongs to $C^{1,2}((0,+\infty)\times\overline{B(0,n)})\cap C([0,+\infty)\times\overline{B(0,n)})$,
is bounded in each strip $[0,T]\times\overline{B(0,n)}$ and
\begin{align*}
\left\{
\begin{array}{ll}
\displaystyle D_tv_n=\mathcal A v_n+pv_n^{1-\frac{2}{p}}(\widetilde\psi_1+\widetilde\psi_2+\widetilde\psi_3)+pv_n^{1-\frac{4}{p}}\widetilde\psi_4,\quad & {\rm in}~ (0,1]\times B(0,n),\\[1mm]
\displaystyle\frac{\partial v_n}{\partial \nu}(t,x)\le 0, & {\rm in}~(0,1]\times\partial B(0,n), \\[1mm]
v_n(0,\cdot)=(|\f|^2+\varepsilon_p)^{\frac{p}{2}}, & {\rm in}~ B(0,n),
\end{array}
\right.
\end{align*}
with
\begin{align*}
\widetilde\psi_1(t,\cdot) =& -\sum_{k=1}^m|\sqrt{Q}\nabla_xu_{n,k}(t,\cdot)|^2+\gamma t\psi_1(t,\cdot),\\
\widetilde\psi_2(t,\cdot) =&\gamma t\psi_2(t,\cdot),\vspace{2mm}\\
\widetilde\psi_3(t,\cdot) =& \frac\gamma2|J_x\uu_n(t,\cdot)|^2+\sum_{j=1}^d\langle \uu_n(t,\cdot), B_jD_j\uu_n(t,\cdot)\rangle,\\[2mm]
\widetilde\psi_4(t,\cdot)=&(2-p)\sum_{h,k=1}^m\langle Q\zeta_h(t,\cdot),\zeta_k(t,\cdot)\rangle,
\end{align*}
where $\zeta_j=u_j\nabla_xu_{n,j}+\gamma tD^2_xu_{n,j}\nabla_xu_{n,j}$ for $j=1,\ldots,m$ and the functions $\psi_1$ and $\psi_2$ are defined in the proof of Proposition \ref{premier}.

Using Hypothesis \ref{hyp_base}(i)-(ii) and the Young inequality we estimate $\widetilde{\psi}_i$, $i=1,2,3$ in the following way:
\begin{align}
\widetilde \psi_1(t,\cdot)\le & \gamma t\Lambda_{Jb}|J_x\uu_n(t,\cdot)|^2\!-\!\sum_{k=1}^m|\sqrt{Q}\nabla_xu_{n,k}(t,\cdot)|^2\!
-\!\gamma t\sum_{i=1}^d\sum_{k=1}^m\!|\sqrt{Q}\nabla_xD_iu_{n,k}(t,\cdot)|^2\notag\\
:=&\gamma t\Lambda_{Jb}|J_x\uu_n(t,\cdot)|^2\!-\!{\mathcal I}_1(t,\cdot)-\gamma t{\mathcal I}_2(t,\cdot)\notag\\
\le& \gamma t \Lambda_{Jb}|J_x\uu_n(t,\cdot)|^2-\lambda_Q|J_x\uu_n(t,\cdot)|^2-\gamma t\lambda_Q|D^2_x\uu_n(t,\cdot)|^2,
\label{A}\\[2mm]
\widetilde\psi_2(t,\cdot)\le & \gamma t\bigg [\bigg (\sum_{i,j=1}^d|D_jB_i|^2\bigg )^{\frac{1}{2}}\!+\!\frac{(\sqrt{d}m\xi\!+\!d^{\frac{3}{2}}k\lambda_Q^{-\frac{1}{2}})^2}{4a}\bigg ]|J_x\uu_n(t,\cdot)|^2\notag\\
&+a\gamma t\lambda_Q|D_x^2\uu_n(t,\cdot)|^2,
\label{star-2}
\\[2mm]
\widetilde\psi_3(t,\cdot)\leq & \frac{\gamma}{2} |J_x\uu_n(t,\cdot)|^2+\bigg (\sum_{j=1}^d|B_j|^2\bigg )^{\frac{1}{2}}|\uu_n(t,\cdot)||J_x\uu_n(t,\cdot)|\notag\\
\le &\frac{\gamma}{2}|J_x\uu_n(t,\cdot)|^2+\sqrt{d}m\xi\lambda_Q^{\frac{1}{2}}|\uu_n(t,\cdot)||J_x\uu_n(t,\cdot)|\notag\\
\le &\frac{dm^2\xi^2}{4\varepsilon}|\uu_n(t,\cdot)|^2+\bigg (\frac{\gamma}{2}+\varepsilon\lambda_Q\bigg )|J_x\uu_n(t,\cdot)|^2
\label{star-3}
\end{align}
for any $t\in (0,1]$ and $a,\varepsilon>0$, where $\xi$ is defined in Hypothesis \ref{hyp_base}(ii).
We distinguish two cases.
If $p \ge p_0\vee 2$ then $\widetilde{\psi}_4\le 0$. Thus, using the previous estimates with $a=1$, $\varepsilon=2^{-1}$ and $\gamma=\lambda_0$
and Hypothesis \ref{hyp_base}(iv) we obtain immediately that $D_tv_n-\mathcal{A}v_n \le 2^{-1}pdm^2\xi^2v_n^{1-\frac{2}{p}}|\uu_n|^2\le h_p v_n$,
where $h_p=2^{-1}pdm^2\xi^2$. On the other hand, if $p_0<2$ and $p\in (p_0,2)$ then we need to estimate $\widetilde\psi_4$ too. We obtain
\begin{align}
\widetilde\psi_4(t,\cdot)
\le &\bigg (\sum_{h=1}^m|u_{n,h}(t,\cdot)||\sqrt{Q}\nabla u_{n,h}(t,\cdot)|\bigg )^2\notag\\
&+2\gamma t\sum_{h=1}^m|u_{n,h}(t,\cdot)||\sqrt{Q}\nabla u_{n,h}(t,\cdot)|\sum_{h=1}^m|\sqrt{Q}D^2_xu_{n,h}(t,\cdot)\nabla_xu_{n,h}(t,\cdot)|\notag\\
&+\gamma^2t^2\bigg (\sum_{h=1}^m|\sqrt{Q}D^2_xu_{n,h}(t,\cdot)\nabla_xu_{n,h}(t,\cdot)|\bigg )^2\notag\\
\le&\bigg [|\uu_n(t,\cdot)|\bigg (\sum_{h=1}^m|\sqrt{Q}\nabla_xu_{n,h}(t,\cdot)|^2\bigg )^{\frac{1}{2}}\!+\!\gamma t|J_x \uu_n| \bigg (\sum_{h=1}^m|\sqrt{Q}D^2_xu_h(t,\cdot)|^2\bigg )^{\frac{1}{2}}\bigg ]^2\notag\\
= & \big [|\uu_n(t,\cdot)|\sqrt{{\mathcal I}_1(t,\cdot)}+\gamma t|J_x\uu_n(t,\cdot)|\sqrt{{\mathcal I}_2(t,\cdot)}\big ]^2\notag\\
\le & (|\uu_n(t,\cdot)|^2+\gamma t|J_x\uu_n(t,\cdot)|^2)\big ({\mathcal I}_1(t,\cdot)+\gamma t{\mathcal I}_2(t,\cdot)\big )\notag\\
\le & (v_n(t,\cdot))^{\frac{2}{p}}\big ({\mathcal I}_1(t,\cdot)+\gamma t{\mathcal I}_2(t,\cdot)\big )
\label{B}
\end{align}
for any $t\in (0,1]$.
Thus, choosing $a= p-1$, $\gamma=(p-1)\lambda_0^{-1}$ and $\varepsilon=(p-1)/2$, from \eqref{A}-\eqref{B}  we obtain that
\begin{align*}
D_tv_n-\mathcal Av_n\leq
\frac{pdm^2\xi^2}{2(p-1)}|\uu_n(t,\cdot)|^2v_n^{1-\frac{2}{p}}\le \frac{pdm^2\xi^2}{2(p-1)}v_n=:\widetilde{h}_pv_n.
\end{align*}
Now, the procedure is the same in the two cases and, arguing as in the last part of Proposition \ref{premier}, we conclude that
\begin{eqnarray*}
 t^{\frac{p}{2}}|J_x\uu_n(t,\cdot)|^p\le k_pT_n^{\mathcal N}(t)((|\f|^2+\varepsilon_p)^{\frac{p}{2}})
\end{eqnarray*}
in $\Rd$, for any $t\in (0,1]$, where $k_p=2^{-1}(p\wedge 2-1)^{-1}pdm^2\xi^2$.
Thus, (letting $\varepsilon_p$ tend to $0^+$ if $p\in (p_0,2)$ and $p_0<2$),
we deduce \eqref{monomi} and the proof is so completed.
\end{proof}

\subsection{Further properties of the semigroup}
\label{subsect-2.4}
As we have already stressed, for each $\f\in C_b(\Rd;\Rm)$ and $t>0$, the function $\T(t)\f$ is bounded on $\Rd$ and estimate \eqref{pointwise} holds true.
For our purpose, we need to slightly improve Theorem \ref{point_prop}, showing the global in time and local in space boundedness of the function $\T(\cdot)\f$.

For notational convenience, for each $\sigma>0$ we denote by $B_{\sigma}(\Rd;\Rm)$ (resp. $C_{\sigma}(\Rd;\Rm)$) the set of all measurable (resp. continuous) vector-valued functions $\f:\Rd\to\Rm$ such that $\f\varphi^{-\sigma}$ is bounded in $\Rd$, where $\varphi$ is the Lyapunov function in Hypothesis \ref{hyp_base}(iii).
It is a Banach space when endowed with the norm $\|\f\|_{B_{\sigma}(\Rd;\Rm)}={\rm esssup}_{x\in\Rd}|(\varphi(x))^{-\sigma}\f(x)|$ (resp. $\|\f\|_{C_{\sigma}(\Rd;\Rm)}={\sup}_{x\in\Rd}|(\varphi(x))^{-\sigma}\f(x)|$).

\begin{thm}
\label{parodontax}
Let Hypotheses $\ref{hyp_base}$ hold true. Then,
\begin{enumerate}[\rm(i)]
\item
there exists a positive constant $C_0\geq1$ such that
\begin{align}
\label{giroitalia}
|(\T(t)\f)(x)|\leq C_0\|\f\|_{\infty}(\varphi(x))^{\frac{1}{\gamma}},\qquad\;\,t>0,\;\,x\in\Rd,
\end{align}
for any $\f\in B_b(\Rd;\Rm)$, where $\gamma$ is defined in Hypothesis $\ref{hyp_base}(v)$;
\item
for any $\sigma \in (0,1/2]$, $\T(t)$ can be extended to $C_{\sigma}(\Rd;\Rm)$ with a semigroup. More precisely,
there exists a positive constant $C_1\ge 1$ such that
\begin{equation}
\|\T(t)\f\|_{C_{\sigma}(\Rd;\Rm)}\le C_1e^{\beta t}\|\f\|_{C_{\sigma}(\Rd;\Rm)},\qquad\;\,t>0,\;\,x\in\Rd,
\label{1200-gs}
\end{equation}
for any $\f\in C_{\sigma}(\Rd;\Rm)$, where $\beta$ is the constant in Theorem $\ref{point_prop}$. Finally,
for any $0\le \delta\le \gamma_0:=\min\{1-1/\gamma, 1/p_0\}$ there exists a constant $C_2=C_2(\delta,\gamma)\ge 1$ such that
\begin{align}
\label{giroitalia-1}
|(\T(t)\f)(x)|\leq C_2\|\f\|_{B_{\delta}(\Rd;\Rm)}(\varphi(x))^{\delta+\frac{1}{\gamma}},\qquad\;\,t>0,\;\,x\in\Rd,
\end{align}
for any $\f\in B_{\delta}(\Rd;\Rm)$.
\end{enumerate}
\end{thm}

\begin{proof}
(i) Since it is rather long, we split the proof into some steps.

{\em Step 1.} As a starting point, we prove that, if $v\in C^{1,2}((0,+\infty)\times\Rd)\cap C([0,+\infty)\times\Rd)$ is bounded in each strip $[0,T]\times\Rd$ and solves the Cauchy
problem
\begin{equation}
\left\{
\begin{array}{lll}
D_tv(t,x)={\mathcal A}v(t,x)+g(t,x), \quad &t>0, &x\in\Rd,\\[1mm]
v(0,x)=f_0(x), &&x\in\Rd,
\end{array}
\right.
\label{lanotte}
\end{equation}
for some functions $f_0\in C_b(\Rd)$ and $g$ such that the function $(s,x)\mapsto \sqrt{s}g(s,x)$ is bounded and continuous in $[0,T]\times\Rd$, then
\begin{eqnarray*}
v(t,x)=(T(t,\cdot)v_0)(x)+\int_0^t(T(t-s)g(s,\cdot))(x)ds,\qquad\;\,t>0,\;\,x\in\Rd.
\end{eqnarray*}

For this purpose, we observe that Hypothesis \ref{hyp_base}(iii) yields a maximum principle for solutions to the Cauchy problem \eqref{lanotte} which belong to $C^{1,2}((0,+\infty)\times\Rd)\cap C([0,+\infty)\times\Rd)$ and are bounded in each strip $[0,T]\times\Rd$. Hence, $v$ is the unique solution to problem \eqref{lanotte}. Up to splitting $g$ into its positive and negative part, we can assume that $g$ is nonnegative on $(0,T]\times\Rd$. By classical results, the Cauchy-Dirichlet problem
\begin{eqnarray*}
\left\{
\begin{array}{lll}
D_tv(t,x)={\mathcal A}v(t,x)+g(t,x),\quad &t>0, &x\in B(0,n),\\[1mm]
v(t,x)=0, & t>0, & x\in\partial B(0,n),\\[1mm]
v(0,x)=f_0(x), &&x\in B(0,n),
\end{array}
\right.
\end{eqnarray*}
admits, for any $n\in\N$, a unique solution $v_n\in C^{1,2}((0,+\infty)\times B(0,n))$ which is bounded and continuous in
$([0,+\infty)\times\overline{B(0,n)})\setminus (\{0\}\times\partial B(0,n))$. In particular, each function $v_n$ is nonnegative
in $(0,+\infty)\times B(0,n)$. Hence, applying the classical maximum principle to the function $v_{n+1}-v_n$,
we easily deduce that the sequence $(v_n)$ is pointwise increasing in $B(0,n)$. Moreover, since
\begin{eqnarray*}
v_n(t,x)=(T_n(t,\cdot)f_0)(x)+\int_0^t(T_n(t-s)g(s,\cdot))(x)ds,\qquad\;\,t>0,\;\,x\in B(0,n),
\end{eqnarray*}
where $T_n(t)$ is the analytic semigroup of contractions in $C_b(B(0,n))$ associated with the realization of the operator ${\mathcal A}$ with homogeneous Dirichlet boundary conditions,
we can estimate
\begin{eqnarray*}
|v_n(t,x)-(T_n(t)f_0)(x)|\le 2\sqrt{t}\sup_{s\in (0,t]}\sqrt{s}\|g(s,\cdot)\|_{\infty}
\end{eqnarray*}
for any $t>0$ and $n\in\N$. Clearly, the function $v_0$, which is the pointwise limit of the sequence $(v_n)$, fulfills the same estimate, so that
\begin{eqnarray*}
|v_0(t,x)|\le \|f_0\|_{\infty}+2\sqrt{T}\sup_{s\in (0,T]}\sqrt{s}\|g(s,\cdot)\|_{\infty}
\end{eqnarray*}
for any $(t,x)\in [0,T]\times\Rd$ and $T>0$, and
\begin{eqnarray*}
|v_0(t,x)-(T(t)f_0)(x)|\le 2\sqrt{t}\sup_{s\in (0,1]}\sqrt{s}\|g(s,\cdot)\|_{\infty}
\end{eqnarray*}
for any $(t,x)\in [0,1]\times\Rd$. Since the function $T(\cdot)f_0$ is continuous in $[0,+\infty)\times\Rd$, the above estimate shows that $v$ can be extended by continuity in $\{0\}\times\Rd$, where
it equals function $f_0$. To identify $v_0$ with $v$ it suffices we use the interior Schauder estimates in Theorem \ref{thm-A2} and the uniform $L^{\infty}$-boundedness of $v_n$, to infer that
the sequence $(v_n)$ is bounded in $C^{1+\alpha/2,2+\alpha}(K)$ for any compact set $K\subset (0,+\infty)\times\Rd$. Arzel\`a-Ascoli theorem and the pointwise convergence of $v_n$ to $v_0$ show that
$v_n$ converges to $v_0$ in $C^{1,2}(K)$ for any $K$ as above, so that, in particular, $v_0\in C^{1,2}((0,+\infty)\times\Rd)$ and solves the Cauchy problem \eqref{lanotte}. Thus, $v=v_0$.

{\em Step 2.}
Here, based on Step 1, we show that
\begin{equation}
(\T(t)\f)_i(x)=(T(t)f_i)(x)+\int_0^t(T(t-s)w_i(s,\cdot))(x)ds,
\label{divano}
\end{equation}
with $w_i=\sum_{j=1}^d\sum_{h=1}^m(B_j)_{ih}D_ju_h$ for any $(t,x)\in [0,+\infty)\times\Rd,\;\,i=1,\ldots,m$. For this purpose, we fix a sequence $(\vartheta_n)$ of odd and smooth enough functions such that, for any $n\in\N$, $\vartheta_n(t)=t$ if $0\le t \leq n$, $\vartheta_n(t)=n+1/2$
if $t\geq n+1$, $0\le\vartheta_n'\le 1$ in $\R$  and $\vartheta_n''\leq 0$ in $[0,+\infty)$.
Then, we consider the Cauchy problem
\eqref{alimentari}, where now the operator $\A$ is replaced by the operator $\A_n$ defined as $\A$, with the matrices $B_i$ being replaced by the matrices $B_{i,n}$, with entries
$(B_{i,n})_{hk}=\vartheta_n\circ(B_{i,n})_{hk}$ .
Clearly, $|(B_{i,n})_{hk}|\le |(B_i)_{hk}|\le \xi\sqrt{\lambda_Q}$ in $\Rd$, for any $n\in\N$, $i=1,\ldots,d$ and $h,k=1,\ldots,m$, so that Theorem 2.9 in \cite{AALT} applies and shows
that, for any $n\in\N$, there exists a unique function $\uu_n\in C([0,+\infty)\times\Rd;\Rm)\cap C^{1,2}((0,+\infty)\times\Rd;\Rm)$, which is bounded in each strip $[0,T]\times\Rd$, solves the equation
$D_t\uu_n=\A_n\uu_n$ on $(0,+\infty)\times\Rd$ and agrees with the function $\f$ on $\{0\}\times\Rd$. As a first step, we observe that, up to a subsequence,
$\uu_n$ converges to a function $\vv$ in $C^{1,2}(K)$ for any compact set $K\subset (0,+\infty)\times\Rd$. Indeed, by Theorem \eqref{point_prop}, the sequence $(\uu_n)$ is bounded in $[0,T]\times\Rd$ for any $T>0$ and, thus, the interior Schauder estimates in Theorem \ref{thm-A2} show
it is bounded in $C^{1+\alpha/2,2+\alpha}(K)$ for any $K$ as above,  Hence, we can argue as in the last part of Step 1. In particular, it turns out the function $\vv$ solves the differential equation $D_t\vv=\A\vv$ in $(0,+\infty)\times\Rd$ and is bounded in each strip $[0,T]\times\Rd$.

Next, we observe that, by Proposition \ref{bundes}, which can be applied also in this situation
since $|D_jB_{i,n}|\le \|\vartheta_n'\|_{\infty}|D_jB_i|\le |D_jB_i|$ for any $i,j=1,\ldots,d$ and $n\in\N$, we deduce that
\begin{eqnarray*}
|J_x\uu_n(t,x)|^2\le c_T e^{2\sigma_{p_0}t}t^{-1}((T(s)|\f|^{p_0})(x))^{\frac{2}{p_0}}
\end{eqnarray*}
for any $t\in (0,T]$, $x\in\Rd$, $T>0$ and some positive constant $c_T$ depending also on $p_0$.

In view of the previous estimate and Step 1, we can write
\begin{equation*}
u_{n,i}(t,x)=(T(t,\cdot)f_i)(x)+\int_0^t(T(t-s)w_{n,i}(s,\cdot))(x)ds
\end{equation*}
for any $t>0$, $x\in\Rd$, $n\in \N$ and $i=1, \ldots,m$, where $w_{n,i}$ ($i=1,\ldots,d$) is defined as $w_i$, with the matrices $B_j$ being replaced by the matrices $B_{j,n}$ ($n\in\N$).
Clearly, the function $(r,s,x)\mapsto (T(r)w_{n,j}(s,\cdot))(x)$ is continuous
on $(0,+\infty)\times (0,+\infty)\times\Rd$ and, in view of Hypothesis \ref{hyp_base}(ii), we can estimate
\begin{align}
|w_{n,i}(s,x)|
\leq & c_Ts^{-\frac{1}{2}}\|\f\|_{\infty}\psi(x),\qquad\;\,s\in (0,T],\;\,x\in\Rd,
\label{addio-0}
\end{align}
for some positive constant $c_T$ depending also on $p_0$. Hypothesis \ref{hyp_base}(v), the H\"older inequality, formula \eqref{valigia} and estimate \eqref{addio-00} show that
\begin{align}
(T(t)\psi)(x)=&\int_{\Rd}\psi p(t,x,dy)\le \bigg (\int_{\Rd}\psi^\gamma p(t,x,dy)\bigg )^{\frac{1}{\gamma}}\le c_{\gamma}^{\frac{1}{\gamma}}\bigg (\int_{\Rd}\varphi p(t,x,dy)\bigg )^{\frac{1}{\gamma}}\notag\\
=&c_{\gamma}^{\frac{1}{\gamma}}((T(t)\varphi)(x))^{\frac{1}{\gamma}}
\le c_{\gamma}^{\frac{1}{\gamma}}\bigg (\frac{a_*}{c_*}+\varphi(x)\bigg )^{\frac{1}{\gamma}}
\label{addio}
\end{align}
for any $t>0$ and $x\in\Rd$. Taking into account that $\varphi\ge 1$ in $\Rd$, we conclude that $T(\cdot)\psi\le c\varphi^{1/\gamma}$ in $(0,+\infty)\times\Rd$. In particular, $T(\cdot)\psi$ is bounded in
$(0,+\infty)\times B(0,r)$ for any $r>0$. Hence, taking also Theorem \ref{maremaremare} into account, we can apply twice the dominated convergence theorem to show, first, that
$T(t-\cdot)w_{n,i}$
pointwise converges to $T(t-\cdot)\widetilde w_i$ (where $\widetilde w_i$ is defined as $w_i$ ($i=1,\ldots,m$) with $\uu$ being replaced by $\vv$) and then that
\begin{eqnarray*}
\lim_{n\to +\infty}\int_0^t(T(t-s)w_{n,i}(s,\cdot))(x)ds
=\int_0^t(T(t-s)\widetilde w_i(s,\cdot))(x)ds
\end{eqnarray*}
for any $t>0$ and $x\in\Rd$. It thus follows that
\begin{eqnarray*}
v_i(t,x)=(T(t)f_i)(x)\!+\!\int_0^t(T(t-s)w_i(s,\cdot))(x)ds,\qquad\;\, t>0,\;\, x\in\Rd,\;\,i=1,\ldots,m.
\end{eqnarray*}
Using estimates \eqref{addio-0} and \eqref{addio} we conclude that
\begin{align*}
|v_i(t,x)-f_i(x)|\le &|(T(t)f_i)(x)-f_i(x)|+c\|\f\|_{\infty}\int_0^ts^{-\frac{1}{2}}(T(t-s)\psi)(x)ds\\
\le &|(T(t)f_i)(x)-f_i(x)|+c\|\f\|_{\infty}(\varphi(x))^{\frac{1}{\gamma}}\sqrt{t}
\end{align*}
for any $t\in (0,1]$, $x\in\Rd$, $i=1,\ldots,m$ and some positive constant $c$ depending on $d,m$ and $p_0$.
From this chain of inequalities we easily deduce that $\vv$ is continuous on $\{0\}\times\Rd$, where it equals the function $\f$.
Summing up, we have shown that $\vv\in C^{1,2}((0,+\infty)\times\Rd;\Rm)\cap C([0,+\infty)\times\Rd;\Rm)$, solves the differential equation
$D_t\vv=\A\vv$ in $(0,+\infty)\times\Rd$ and $\vv(0,\cdot)=\f$. By Theorem \ref{point_prop}, we conclude that $\vv=\uu$ and formula \eqref{divano} follows.

{\em Step 3.} Using \eqref{capitanamerica} and \eqref{divano}, we can estimate
\begin{align*}
w_i(s,x)\le &\sqrt{dm}\psi(x)|J_x\T(s)\f)(x)|
\le c_2(s^{-\frac{1}{2}}\vee 1)e^{\sigma_2s}\|\f\|_{\infty}\psi(x)
\end{align*}
for any $s>0$, $x\in\Rd$, $i=1,\ldots,m$. Hence, for $t>0$, $x\in\Rd$ and $i=1,\ldots,m$ we get
\begin{align*}
|u_i(t,x)|\le &|(T(t)f_i)(x)|+\int_0^t|(T(t-s)w_i(s,\cdot))(x)|ds\\
\le & \|f_i\|_{\infty}+c\|\f\|_{\infty}(\varphi(x))^{\frac{1}{\gamma}}\int_0^t(s^{-\frac{1}{2}}\vee 1)e^{\sigma_{p_0}s}ds\\
\le & \|f_i\|_{\infty}+c\|\f\|_{\infty}(\varphi(x))^{\frac{1}{\gamma}}.
\end{align*}
Estimate \eqref{giroitalia} follows at once for functions in $C_b(\Rd;\Rm)$.

Suppose now that $\f\in B_b(\Rd;\Rm)$ and let $(\f_n)\subset C_b(\Rd;\Rm)$ be a bounded sequence converging to $\f$ almost everywhere in $\Rd$, with
respect to the Lebesgue measure, and such that $\|\f_n\|_{\infty}\le\|\f\|_{\infty}$ for any $n\in\N$. Then, \cite[Corollary 3.4]{AALT} shows that $\T(\cdot)\f_n$ converges to $\T(\cdot)\f$ pointwise in $(0,+\infty)\times\Rd$, as $n$ tends to $+\infty$. Writing \eqref{giroitalia} with $\f$ being replaced by $\f_n$ and letting $n$ tend to $+\infty$, we complete the proof of \eqref{giroitalia}.

(ii) Fix $\sigma \in (0,1/2]$. Without loss of generality, we can assume that all the components of $\f\in C_{\sigma}(\Rd;\Rm)$ are nonnegative since the general case then will follow
splitting $\f=\f^+-\f^-$, where the $i$-th component of $\f^+$ (resp. $\f^-$) is the positive part of $f_i$ (resp. $-f_i$).

For any $n\in\N$, we set $\f_n:=\vartheta_n\f$, where $(\vartheta_n)$ is
a sequence of smooth enough functions satisfying $\chi_{B(0,n)}\leq \vartheta_n\leq \chi_{B(0,2n)}$. We also fix $i,j\in\{1,\ldots,m\}$, $t>0$, $x\in\Rd$ and denote by $P=P_{ij}^{t,x}$
the positive set of the Hahn decomposition of $p_{ij}(t,x,dy)$. Since each sequence $(f_{n,j})$ is weakly increasing, by monotone convergence we can infer that
\begin{align*}
\lim_{n\to +\infty}\int_{\Rd}f_{n,j}(y)\chi_{P}(y)p_{ij}(t,x,dy)=\int_{\Rd}f_j(y)\chi_{P}(y)p_{ij}(t,x,dy).
\end{align*}
Moreover, as it has been noticed in Section \ref{sect-2}, the semigroup $\T(t)$ can be extended to $B_b(\Rd;\Rm)$ through formula \eqref{domanda} and
$|\T(t)\f|\le e^{\beta t}(T(t)|\f|^2)^{1/2}$, pointwise in $\Rd$ for any $\f\in B_b(\Rd;\Rm)$, where $\beta$ is the constant in Theorem \ref{point_prop}.
In particular, since
\begin{align*}
\int_{\Rd}f_{n,j}(y)\chi_{P}(y)p_{ij}(t,x,dy)=&|(\T(t)(f_{n,j}\chi_{P}\e_j))_i(x)|\le e^{\beta t}(T(t)|\vartheta_n f_{j}|^2(x))^{\frac{1}{2}}\\
\le & e^{\beta t} \bigg [\bigg (T(t)\bigg (\frac{|\vartheta_n f_{j}|^2}{\varphi^{2\sigma}}\bigg )^{\frac{1}{1-2\sigma}}\bigg )(x)\bigg ]^{\frac{1}{2}-\sigma}((T(t)\varphi)(x))^{\sigma}\\
\le &e^{\beta t}\|\f\|_{C_{\sigma}(\Rd;\Rm)}((T(t)\varphi)(x))^{\sigma}\\
\le &e^{\beta t}\|\f\|_{C_{\sigma}(\Rd;\Rm)}\bigg (\frac{a_*}{c_*}+\varphi(x)\bigg)^{\sigma}\\
\le & ce^{\beta t}\|\f\|_{C_{\sigma}(\Rd;\Rm)}(\varphi(x))^{\sigma},
\end{align*}
we conclude that $\int_{\Rd}f_j\chi_P p_{ij}(t,x,dy)$
is real and
\begin{equation}
\int_{\Rd}f_j(y)\chi_{P}(y)p_{ij}(t,x,dy)\le ce^{\beta t}\|\f\|_{C_{\sigma}(\Rd;\Rm)}(\varphi(x))^{\sigma}.
\label{pannocchie}
\end{equation}
The same arguments can be applied to show that
\begin{align}
\int_{\Rd}f_j\chi_N p_{ij}(t,x,dy)\le ce^{\beta t}\|\f\|_{C_{\sigma}(\Rd;\Rm)}(\varphi(x))^{\sigma},
\label{pannocchie-1}
\end{align}
where $N=N_{ij}^{t,x}$ is the negative set of the Hahn decomposition of the measure $p_{ij}(t,x,dy)$.
In particular, the interior Schauder estimates in Theorem \ref{thm-A2} can be used to prove that the function $\T(\cdot)\f$ is continuous in $\Rd$ and,
together with \eqref{pannocchie} and \eqref{pannocchie-1}, they allow to conclude that each operator is bounded from $C_{\sigma}(\Rd;\Rm)$ into itself and estimate \eqref{1200-gs} holds true.
To prove the semigroup rule, we observe that
$\T(t)\f_n=\T(t-s)\T(s)\f_n$ in $\Rd$ for any $n\in\N$ and $0<s<t$. Moreover, since $|\f_n|+|\T(s)\f_n|\leq c\varphi^{\sigma}$ in $\Rd$, for any $n\in\N$, $s>0$, by the dominated convergence theorem we conclude that $\T(t)\f=\T(t-s)\T(s)\f$.

Finally, estimate \eqref{giroitalia-1} can be obtained adapting the arguments used in the proof of (i), taking the positivity of $T(t)$ into account. More precisely, using \eqref{capitanfuturo} we can estimate
\begin{align*}
|u_{n,i}(t,\cdot)|\le &|T(t)f_i|+\int_0^t (T(t-s)|w_{n,i}(s,\cdot)|)(\cdot)ds\notag\\
\le &\|\f\|_{C_{\delta}(\Rd;\Rm)}T(t)\varphi^\delta\\
&+ c_{p_0}\int_0^t (s^{-\frac{1}{2}}\vee 1)e^{\sigma_{p_0}s}[T(t-s)(\psi (T(s)|\f|^{p_0})^{\frac{1}{p_0}})](\cdot)ds
\end{align*}
in $\Rd$ for any $\f\in C_{\delta}(\Rd;\Rm)$. Observe that for any $s>0$
\begin{align*}
T(s)|\f|^{p_0}\le \|\f\|_{C_{\delta}(\Rd;\Rm)}^{p_0}T(s)\varphi^{\delta p_0}\le c\|\f\|_{C_{\delta}(\Rd;\Rm)}^{p_0}\varphi^{\delta p_0},
\end{align*}
where in the last inequality we used the fact that $\delta p_0\le 1$ and $T(t)\varphi^\eta\le (T(t)\varphi)^\eta\le (c^{-1}_*a_*+\varphi)^\eta\le\varphi^{\eta}$ for any $t\ge 0$ and $\eta\le 1$
Hence, using the previous estimates, Hypothesis \ref{hyp_base}(v), again the positivity of $T(t)$ and estimate \eqref{addio-00}, we can infer that
\begin{align*}
|u_{n,i}(t,\cdot)|\le &\|\f\|_{C_{\delta}(\Rd;\Rm)}T(t)\varphi^\delta\\
&+c_{p_0}\|\f\|_{C_{\delta}(\Rd;\Rm)}\int_0^t (s^{-\frac{1}{2}}\vee 1)e^{\sigma_{p_0}s}T(t-s)\varphi^{\delta+\frac{1}{\gamma}} ds\notag\\
\le &c_{p_0}\|\f\|_{C_{\delta}(\Rd;\Rm)}\bigg (\varphi^{\delta}+\varphi^{\delta+\frac{1}{\gamma}}\int_0^t(s^{-\frac{1}{2}}\vee 1)e^{\sigma_{p_0}s}ds\bigg )
\end{align*}
for any $t>0$. Letting $n$ tend to $+\infty$, estimate \eqref{giroitalia-1} follows for functions $f\in C_{\delta}(\Rd;\Rm)$.

If $\f$ is not continuous, then it suffices to approximate\footnote{This can be easily done, approximating the bounded function $\f/\varphi^{\delta}$ with a bounded sequence $(\widetilde\f_n)\subset C_b(\Rd;\Rm)$ converging to $\f/\varphi^{\delta}$ almost everywhere in $\Rd$ with respect to the Lebesgue measure and, hence, with respect to each measure $p(t,x,dy)$. Setting $\f_n=\widetilde \f_n \varphi^{\delta}$ we obtain the sought for sequence.} it with a sequence $(\f_n)$ of continuous functions, converging to $\f$ almost everywhere in $\Rd$ and such that
$\sup_{x\in\Rd}|\f_n(x)(\varphi(x))^{-\delta}|\le {\rm esssup}_{x\in\Rd}|\f(x)(\varphi(x))^{-\delta}|$, and use the dominated convergence together with the above result which shows that $\int_{\Rd}\varphi^{1/\gamma}|p_{ij}(t,x,dy)|<+\infty$ for any $t>0$, $x\in\Rd$, $i,j=1,\ldots,m$, to infer that $\T(\cdot)\f_n$ converges to $\T(\cdot)\f$ as $n$ tends to $+\infty$ pointwise in $(0,+\infty)\times\Rd$. Writing \eqref{giroitalia-1} with $\f_n$ replacing $\f$ and letting $n$ tend to $+\infty$, \eqref{giroitalia-1} follows in its full generality.
\end{proof}

\begin{prop}
\label{carovana}
Let $(\f_n)$ be a bounded sequence in $C_{\gamma^{-1}}(\Rd;\Rm)$ which converges to a function $\f\in C(\Rd;\Rm)$, locally uniformly in $\Rd$. Then,
$(\T(\cdot)\f_{n})$ converges uniformly in $(0,+\infty)\times B(0,R)$ to $\T(\cdot)\f$,
for any $R>0$.
\end{prop}

\begin{proof}
We fix $r>0$, set $\g_n:=\f_n-\f$ and notice that $\f,\g_n\in C_{\gamma^{-1}}(\Rd;\Rm)$ for any $n\in\N$.
By Theorem \ref{parodontax}, the functions $\T(t)\f$ and $\T(t)\g_n$ are well defined for any $t>0$ and $n\in\N$.
Moreover, the arguments in Step 2 of the proof of Theorem \ref{parodontax} can be easily adapted to prove, by an approximation argument, that
\begin{eqnarray*}
(\T(t)\g)_i(x)=(T(t)g_i)(x)+\int_0^t\bigg (T(t-s)\sum_{j=1}^d\sum_{h=1}^m(B_j)_{ih}D_j(\T(s)\g)_h\bigg )(x)ds,
\end{eqnarray*}
for any $t>0$, $i=1,\ldots,m$ and $x\in\Rd$, also for any $\g\in C_{1/\gamma}(\Rd;\Rm)$. Applying this formula with $\g=\g_n$ and using \eqref{capitanfuturo} with\footnote{Note
 that such an inequality can be extended to functions in $C_{\gamma^{-1}}(\Rd;\Rm)$ by a density argument, approximating any such function ${\bf h}$ with a sequence
of bounded and continuous functions, which is bounded in $C_{\gamma^{-1}}(\Rd;\Rm)$ and converges to ${\bf h}$ locally uniformly in $\Rd$.} $p=2$, we can infer that
\begin{align*}
&|(\T(t)\f_n)_i(x)-(\T(t)\f)_i(x)|\\
\leq & |(T(t)g_{n,i})(x)|
+c\int_0^t(s^{-\frac{1}{2}}\vee 1)e^{\sigma_{2}s}[T(t-s)(\psi(T(s)|\g_n|^{2})^{\frac{1}{2}})](x)ds \\
\leq & |(T(t)g_{n,i})(x)|
+c|(T(t)|\g_n|^{2})(x)|^{\frac{1}{2}}\int_0^t(s^{-\frac{1}{2}}\vee 1)e^{\sigma_{2}s}|(T(t-s)\varphi)(x)|^{\frac{1}{2}}ds\\
\leq & |(T(t)g_{n,i})(x)|
+c|(T(t)|\g_n|^2)(x)|^{\frac{1}{2}}\sqrt{\varphi(x)}\int_0^{+\infty}(s^{-\frac{1}{2}}\vee 1)e^{\sigma_{2}s}ds
\end{align*}
for any $t>0$ and $x\in\Rd$. Now, we fix $R>0$, $x\in B(0,R)$ and for any $r>0$ we split (see \eqref{valigia})
\begin{align*}
(T(t)|\g_n|^2)(x)
=& \int_{B(0,r)}|\g_n|^2 p(t,x,dy)
+\int_{\Rd\setminus B(0,r)}|\g_n|^2 p(t,x,dy)\notag\\
\le &\|\g_n\|_{C_b(B(0,r))}^2
+\sup_{n\in\N}\|\g_n\|_{C_{\gamma^{-1}}(\Rd;\Rm)}^2\int_{\Rd\setminus B(0,r)}\varphi^{\frac{2}{\gamma}}p(t,x,dy)\\
\le &\|\g_n\|_{C_b(B(0,r))}^2
+\frac{\sup_{n\in\N}\|\g_n\|_{C_{\gamma^{-1}}(\Rd;\Rm)}^2}{\inf_{\Rd\setminus B(0,r)}\varphi^{1-\frac{2}{\gamma}}}\int_{\Rd\setminus B(0,r)}\varphi p(t,x,dy)\\
\le &\|\g_n\|_{C_b(B(0,r))}^2+\frac{c}{\inf_{\Rd\setminus B(0,r)}\varphi^{1-\frac{2}{\gamma}}}\left (\frac{a_*}{c_*}+\sup_{B(0,R)}\varphi\right )
\end{align*}
for any $t>0$ and $n\in\N$. Letting first $n$ and then $r$ tend to $+\infty$ in the first- and last-side of the previous chain of inequalities, taking into account
that $\varphi$ blows up as $|x|$ tends to $+\infty$, we easily conclude that $T(t)|\g_n|^2$ vanishes uniformly in $(0,+\infty)\times B(0,R)$ for any $R>0$.

Finally, since
\begin{align*}
|(T(t)g_{n,i})(x)|\le &(T(t)|\g_n|)(x)=\int_{\Rd}|\g_n|p(t,x,dy)\\
\le &\bigg (\int_{\Rd}|\g_n|^2p(t,x,dy)\bigg )^{\frac{1}{2}}(p(t,x,\Rd))^{\frac{1}{2}}
=|(T(t)|\g_n|^2)(x)|^{\frac{1}{2}}
\end{align*}
for any $t>0$, $x\in\Rd$, $n\in\N$ and $i=1,\ldots,m$,
where we have taken into account that the $p(t,x,dy)$'s are probability measures, we also conclude
that $T(\cdot)|\g_n|$ vanishes uniformly in $(0,+\infty)\times B(0,R)$ for any $R>0$.
\end{proof}

\section{Systems of invariant measures}
\label{sect-3}

\begin{defi}
A family of signed finite Borel measures on $\Rd$ $\{\mu_j: j=1,\ldots,m\}$ is a system  of invariant measures for $\T(t)$
if for any $\f\in C_b(\Rd;\Rm)$ and $t>0$ it holds that
\begin{align}
\sum_{j=1}^m\int_{\Rd}(\T(t)\f)_jd\mu_j=\sum_{j=1}^m\int_{\Rd}f_jd\mu_j.
\label{caricatore}
\end{align}
\end{defi}

By using the continuity properties of the semigroup $\T(t)$ proved in \cite[Corollary 3.4]{AALT} and the dominated convergence theorem, it follows immediately that formula \eqref{caricatore} holds true for any $\f\in C_b(\Rd;\Rm)$ if and only if it is satisfied by any $\f\in B_b(\Rd;\Rm)$.
Moreover,

\begin{prop}
\label{prop-benevento}
A family $\{\mu_j: j=1,\ldots,m\}$ of $($signed$)$ finite measures is a system of invariant measures for $\T(t)$ if and only if
\begin{align}
\label{trasportino}
\sum_{i=1}^m\int_{\R^d}(\A \f)_id\mu_i=0, \qquad \f\in D_{\max}(\A).
\end{align}
\end{prop}

\begin{proof}
First, we suppose that $\{\mu_i:i=1, \ldots, m\}$ is a system of invariant measures for $\T(t)$ and fix $\f\in D_{\max}(\A)$.
The invariance property of the system $\{\mu_j: j=1,\ldots,m\}$ implies that
\begin{align}
\sum_{j=1}^m\int_{\Rd}\frac{(\T(t)\f)_j-f_j}{t}d\mu_j=0,\qquad\;\,t>0.
\label{tivano-1}
\end{align}
By Proposition \ref{incarico} we know that, for any $j=1,\ldots,m$, $t^{-1}((\T(t)\f)_j-f_j)$ converges to $({\mathcal A}\f)_j$ pointwise in $\Rd$ as $t$ tends to $0^+$.
Moreover, $\sup_{t\in (0,1]}t^{-1}|(\T(t)\f)_j-f_j|$ is a bounded function in $\Rd$, thanks to Proposition \ref{incarico}. Since each $\mu_j$ is a finite measure,
we can let $t$ tend to $0^+$ in both sides of \eqref{tivano-1} and obtain \eqref{trasportino}.

Let us now assume that \eqref{trasportino} holds true in $D_{\max}(\A)$ and fix $\f$ in such a space. Then,
\begin{eqnarray*}
(\T(t)\f)_i(x)-f_i(x)=\int_0^t(\T(s)\A\f)_i(x)ds,\qquad\;\,t>0,\;\,x\in\Rd,\;\,i=1,\ldots,m.
\end{eqnarray*}
Therefore, integrating again in $\Rd$ with respect to the measure $\mu_i$, summing over $i$ from $1$ to $m$ and applying the Fubini's theorem, we deduce that
\begin{align*}
\sum_{i=1}^m\int_{\Rd}((\T(t)\f)_i\!-\!f_i)d\mu_i&=
\sum_{i=1}^m\int_{\Rd}d\mu_i\!\int_0^t(\T(s)\A\f)_ids\\
&= \int_0^t\bigg (\sum_{i=1}^m\int_{\Rd}(\A\T(s)\f)_id\mu_i\bigg )ds=0,
\end{align*}
and this completes the proof.
\end{proof}

\medskip

Under Hypotheses \ref{hyp_base} we prove that there exist $m$-systems of invariant measures for $\T(t)$.

The following result shows that the average in $(0,t)$ of any component of the function $\T(t)\f$ converges
as $t$ tends to $+\infty$. As in the scalar case, this convergence allows us to define the systems of invariant measures
associated to $\{\T(t)\}_{t\geq0}$ (see \cite[Prop. 8.1.13]{newbook}).

\begin{thm}
\label{succhinopera}
Assume that Hypotheses $\ref{hyp_base}$ hold true. Then, there exist at least $m$ systems $\{\mu^i_j: j=1, \ldots,m\}$, $i=1, \ldots, m$, of invariant measures associated to the semigroup $\T(t)$ in $C_b(\Rd;\Rm)$.
\end{thm}

\begin{proof}
We split the proof into several steps.

{\em Step 1.} Here, we introduce the sequence $(\boldsymbol{\mathcal R}_n)$ of bounded linear operators in $C_b(\Rd;\Rm)$ defined by
\begin{eqnarray*}
(\boldsymbol{\mathcal R}_n\f)(x)=\frac{1}{n}\sum_{k=0}^{n-1}(\T(k)\f)(x),\qquad\;\,x\in\Rd,\;\,\f\in C_b(\Rd;\Rm),
\end{eqnarray*}
and prove that, for any $\f\in C_b(\Rd;\Rm)$, $\boldsymbol{\mathcal R}_n\f$ converges locally uniformly in $\Rd$ as $n$ tends to $+\infty$ to a constant function.
We fix any such function $\f$ and we first show that a suitable subsequence
of $(\boldsymbol{\mathcal R}_n\f)$ converges locally uniformly in $\Rd$. For this purpose,
we observe that $\boldsymbol{\mathcal R}_n\f=
n^{-1}\f+\T(1)((1-n^{-1})\boldsymbol{\mathcal R}_{n-1}\f)$ in $\Rd$.
By Theorem \ref{parodontax}(i), the sequence $(\boldsymbol{\mathcal R}_{n-1}\f)$ is bounded in $C_{\gamma^{-1}}(\Rd;\Rm)$. Hence,
we can determine a subsequence $(\T(1)((1-n_k^{-1})\boldsymbol{\mathcal R}_{n_k-1}\f))$, which
converges locally uniformly in $\Rd$ to a function $\g\in C_{\gamma^{-1}}(\Rd;\Rm)$. Indeed, Theorem \ref{parodontax}(i) and the interior Schauder estimates in Theorem \ref{thm-A2} show that
\begin{align*}
\|\T(1)((1-n^{-1})\boldsymbol{\mathcal R}_{n-1}\f)\|_{C^{2+\alpha}(B(0,R);\Rm)}\leq &c_R\|(1-n^{-1})\boldsymbol{\mathcal R}_{n-1}\f\|_{C_b(B(0,R+1);\Rm)}\\
\le &c_R\|\f\|_{\infty}\sup_{B(0,R+1)}\varphi^{\frac{1}{\gamma}}
\end{align*}
for any $R>0$. Thus, the Arzel\`a-Ascoli theorem and a compactness argument allow us to extract a subsequence of $(\boldsymbol{\mathcal R}_n\f)$ converging locally uniformly in $\Rd$ to a function $\g\in C_{\gamma^{-1}}(\Rd;\R^m)$.

To prove that all the sequence $(\boldsymbol{\mathcal R}_n\f)$ converges to $\g$ locally uniformly in $\Rd$, we observe that\footnote{The below limits
are all local uniform in $\Rd$.}
\begin{align*}
\f-\g=&\f-\lim_{k\to+\infty}\boldsymbol{\mathcal R}_{n_k}\f
=\lim_{k\to+\infty}
\frac{1}{n_k}\sum_{j=1}^{n_k-1}(I-(\T(1))^j)\f=\lim_{k\to+\infty}(I-\T(1))\bm\zeta_k,
\end{align*}
where $\bm\zeta_k=n_k^{-1}\sum_{j=1}^{n_k-1}\sum_{h=0}^{j-1}\T(h)\f$ ($k\in\N$) is a bounded and continuous function and the sequence $((I-\T(1))\zeta_k)$ is bounded in
$C_{\gamma^{-1}}(\Rd;\Rm)$.
Moreover, $\boldsymbol{\mathcal R}_n(I-\T(1))\bm\zeta_k=n^{-1}(\bm\zeta_k-\T(n)\bm\zeta_k)$
for any $k,n\in\N$. Combining the last two formulas, we can estimate
\begin{align*}
&\|\boldsymbol{\mathcal R}_n(\f-\g)\|_{C_b(B(0,r);\R^m)}\notag\\
\le &\|\boldsymbol{\mathcal R}_n[\f-\g-(I-\T(1))\bm\zeta_k]\|_{C_b(B(0,r);\R^m)}+\|\boldsymbol{\mathcal R}_n(I-\T(1))\bm\zeta_k\|_{C_b(B(0,r);\R^m)}\notag\\
\le& \|\T(\cdot)[\f-\g-(I-\T(1))\bm\zeta_k]\|_{C_b((0,+\infty)\times B(0,r);\R^m)}+
\frac{1}{n}\|\bm\zeta_k-\T(n)\bm\zeta_k\|_{C_b(B(0,r);\R^m)}
\end{align*}
for any $k,n\in\N$ and $r>0$. Now, letting first $n$ and then $k$ tend to $+\infty$, taking Proposition \ref{carovana} into account, from the above chain of inequalities we can infer that $\boldsymbol{\mathcal R}_n(\f-\g)$ vanishes locally uniformly in $\Rd$ as $n$ tends to $+\infty$.
The convergence of $\boldsymbol{\mathcal R}_n\g$ is easier to prove since $\boldsymbol{\mathcal R}_n\g=\g$ in $\Rd$ for any $n\in\N$. Indeed, since $\boldsymbol{\mathcal R}_{n_k}\f-\T(1)\boldsymbol{\mathcal R}_{n_k}\f= n_k^{-1}(\f-\T(n_k)\f)$ in $\Rd$,
letting $k$ tend to $+\infty$, the last side of the previous equality vanishes locally uniformly in $\Rd$. Moreover, since
$\boldsymbol{\mathcal R}_{n_k}\f$ converges to $\g$ locally uniformly in $\Rd$, by Proposition \ref{carovana} $\T(1)\boldsymbol{\mathcal R}_{n_k}\f$ converges to $\T(1)\g$, locally uniformly
in $\Rd$. Thus, we conclude that $\g=\T(1)\g$ in $\Rd$. Using the semigroup rule in Theorem \ref{parodontax}(ii), we deduce that $\T(k)\g=\g$ in $\Rd$ for any $k\in\N$, which implies the claim.

Finally, we prove that $\g$ is a constant function. For this purpose, we approximate $\g$ locally uniformly on $\Rd$ by a sequence $(\g_n)$ of bounded and continuous functions such that $|\g_n|\le |\g|$ in $\Rd$ for any $n\in\N$.
Thanks to the interior Schauder estimates in Theorem \ref{thm-A2} and Theorem \ref{parodontax}(i) we conclude that the sequence $(\T(\cdot)\g_n)$
is bounded in $C^{1+\alpha/2,2+\alpha}({\mathcal K})$ for any compact set ${\mathcal K}\subset (0,+\infty)\times\Rd$. Hence,
up to a subsequence, $\T(\cdot)\g_n$ converges in $C^{1,2}(\mathcal K)$, for any ${\mathcal K}$ as above,
to a function $\boldsymbol\zeta\in C^{1+\alpha/2,2+\alpha}_{\rm loc}((0,+\infty)\times\Rd)$.
On the other hand, by Proposition \ref{carovana}, $\T(\cdot)\g_n$ converges to $\T(\cdot)\g$ uniformly in $(0,+\infty)\times B(0,R)$,
for any $R>0$. We conclude, in particular, that
$J_x\T(k)\g_n$ converges to $J_x\T(k)\g$ locally uniformly in $\Rd$, for any $k\in\N$.
We are almost done. Indeed, using Proposition \ref{bundes} we can now estimate
\begin{align*}
|J\g|^2
= & |J_x\T(k)\g|^2
=\lim_{n\to +\infty}|J_x\T(k)\g_n|^2\leq  ce^{2\sigma_2k}(1\vee k^{-1})T(k)|\g_n|^2\\
\leq & ce^{2\sigma_2k}(1\vee k^{-1})T(k)|\g|^2\le ce^{2\sigma_2k}\bigg (\frac{a_*}{c_*}+\varphi\bigg ),
\end{align*}
where the convergence is local uniformly in $\Rd$ and we have used \eqref{addio-00} in the last step
of the previous chain of inequalities. We have so shown, that
$\| |J\g| \|_{C(\overline{B(0,R)};\Rm)}\le c_Re^{\sigma_2k}$. Since $\sigma_2<0$, letting $k$
tend to $+\infty$, we conclude that $J\g\equiv 0$ on $B(0,R)$ and, hence, on $\Rd$. This shows that $\g$ is a constant function as claimed.

\emph{Step 2.} Here, we prove that there exist $m$ systems $\{\mu^i_j: i,j=1,\ldots,m\}$ ($i=1,\ldots,m$) of Radon measures such that
\begin{equation}
\lim_{t\to +\infty}(\boldsymbol{\mathcal P}_t\f)_i:=\lim_{t\to +\infty}\frac{1}{t}\int_0^t(\T(s)\f)_ids
=\sum_{j=1}^m\int_{\Rd}f_jd\mu^i_j,
\label{marito}
\end{equation}
locally uniformly in $\Rd$ for any $\f\in C_b(\Rd;\R^m)$ and $i=1,\ldots,m$.
For this purpose, we note that
$\boldsymbol{\mathcal P}_{t}\f=t^{-1}[t]\boldsymbol{\mathcal R}_{[t]}\boldsymbol{\mathcal P}_1\f+t^{-1}\{t\}\boldsymbol{\mathcal P}_{\{t\}}\T([t])\f$ in $\Rd$, for any $t>1$ and $\f\in C_b(\Rd;\Rm)$,
where $[t]$ and $\{t\}$ denote respectively the integer and the fractional part of $t$. Since $|\boldsymbol{\mathcal P}_{\{t\}}\T([t])\f|\le c\|\f\|_{\infty}\varphi^{1/\gamma}$ in $\Rd$, for any $t>0$, due to Theorem \ref{parodontax}(i), letting $t$ tend to $+\infty$ in the above estimate we obtain that $\boldsymbol{\mathcal P}_t\f$ converges locally uniformly on $\Rd$ for every $\f\in C_b(\Rd;\R^m)$ and, in view of Step 1, the limit $\boldsymbol{\mathcal P}_*\f$ is a constant function in $\Rd$.
Thus, it follows that $\boldsymbol{\mathcal P}_*\f=\sum_{i=1}^m({\mathcal M}_i\f)\e_i$
for any $\f\in C_b(\Rd;\Rm)$ and some linear operators ${\mathcal M}_j:C_b(\Rd;\Rm)\to\R$, $j=1,\ldots,m$. Note that these operators are bounded. Indeed, using \eqref{giroitalia}, we can estimate
\begin{eqnarray*}
|(\boldsymbol{\mathcal P}_t\f)(0)|\le \frac{1}{t}\int_0^t|(\T(s)\f)(0)|ds\le c\|\f\|_{\infty}(\varphi(0))^{\frac{1}{\gamma}},\qquad\;\,t>0.
\end{eqnarray*}
Since $(\boldsymbol{\mathcal P}_t\f)(0)$ converges to $\sum_{i=1}^m({\mathcal M}_i\f)\e_i$ as $t$ tends to $+\infty$, we conclude that
$|{\mathcal M}_i\f|\le c(\varphi(0))^{1/\gamma}\|\f\|_{\infty}$ for any $\f\in C_b(\Rd;\Rm)$ and $i=1,\ldots,m$. In particular, each operator ${\mathcal M}_i$ is an
element of $(C_0(\Rd;\Rm))'$ and the Riesz representation theorem shows that there exists a family $\{\mu^i_j: i,j=1,\ldots,m\}$ of finite Radon measures on $\Rd$ such that
\begin{equation}
{\mathcal M}_i\f=\sum_{j=1}^m \int_{\Rd}f_j d\mu^i_j,\qquad\;\,\f\in C_0(\Rd;\R^m).
\label{montezuma}
\end{equation}

To complete the proof of \eqref{marito}, we begin by observing that each operator
${\mathcal M}_j$ is well defined and bounded in $C_b(\Rd;\R^m)$. Moreover, if $\f\in C_b(\Rd;\R^m)$ then we can fix a bounded sequence $(\f_n)\subset C_0(\Rd; \Rm)$, converging to $\f$ locally uniformly in $\Rd$ as $n \to +\infty$, and (taking into account that $|\boldsymbol{\mathcal P}_t(\f_n-\f)|\le\sup_{t\ge 0}|\T(t)(\f_n-\f)|$) estimate
\begin{align*}
\bigg |\boldsymbol{\mathcal P}_t\f\!-\!\sum_{i=1}^m({\mathcal M}_i\f)\e_i\bigg |
\le & |\boldsymbol{\mathcal P}_t\f_n\!-\!\boldsymbol{\mathcal P}_*\f_n|\!+\!\sum_{i,j=1}^m\int_{\Rd}|f_{n,j}\!-\!f_j|d|\mu^i_j|\!+\!\sup_{t\ge 0}|\T(t)(\f_n\!-\!\f)|
\end{align*}
in $\Rd$, for any $t>0$ and $n\in\N$, letting first $t$ and then $n$ tend to $+\infty$, we conclude that
\eqref{montezuma} holds true also for any $\f\in C_b(\Rd;\Rm)$.

{\emph Step 3.} Now, we can complete the proof, showing that, for any $i=1,\ldots,m$, the family $\{\mu^i_j: j=1,\ldots,m\}$ is a system of invariant measures for $\T(t)$.
For this purpose, we fix $\f\in C_b(\Rd;\Rm)$, $\tau>0$, $x\in\Rd$, $i\in\{1,\ldots,m\}$ and observe that
\begin{align}
(\boldsymbol{\mathcal P}_t\T(\tau)\f)(x)=&\frac{1}{t}\int_0^t(\T(s)\T(\tau)\f)(x)ds=\frac{1}{t}\int_0^t(\T(s+\tau)\f)(x)ds\notag\\
=&(\boldsymbol{\mathcal P}_t\f)(x)+\frac{1}{t}\int_t^{t+\tau}(\T(s)\f)(x)ds
-\frac{1}{t}\int_0^{\tau}(\T(s)\f)(x)ds.
\label{invariant-proof}
\end{align}
By Theorem \ref{parodontax}(i), the second and third terms in the right-hand side of \eqref{invariant-proof} can be estimated from above by
$t^{-1}c\|\f\|_{\infty}(\varphi(x))^{1/\gamma}\tau$. Hence, letting $t$ tend to $+\infty$ in both sides of \eqref{invariant-proof}, we deduce that
$(\boldsymbol{\mathcal P}_*\T(\tau)\f)(x)=(\boldsymbol{\mathcal P}_*\f)(x)$ or, equivalently, that
\begin{eqnarray*}
\sum_{i,j=1}^m\bigg  (\int_{\Rd}(\T(s)\f)_jd\mu^i_j\bigg )\e_i=\sum_{i,j=1}^m\bigg (\int_{\Rd}\f_jd\mu^i_j\bigg )\e_i
\end{eqnarray*}
and the assertion follows at once.
\end{proof}

\subsection{Properties of systems of invariant measures}
\label{cmp}

To begin with, we observe that $\mu^i_j(\Rd)=\delta_{ij}$ for any $i,j=1,\ldots,m$, where $\delta_{ij}$ is the Kronecker delta.
Indeed, fix $i,j\in\{1,\ldots,m\}$ and set $\f :=\e_j $. Then, using the invariance property of the system
$\{\mu^i_j: j=1,\ldots,m\}$ we deduce that
\begin{align*}
\mu^i_j(\Rd)
= & \int_{\Rd}f_j\mu^i_j
= \sum_{k=1}^m\int_{\Rd}f_{k}d\mu^i_k
= \lim_{t\to +\infty}\frac1t\int_0^t (\T(s)\f)_i(x)ds=\delta_{ij},
\end{align*}
since $\T(\cdot)\f=\e_j$ in $[0,+\infty)\times\Rd$.

Next, we prove that the total variations of the measures $\mu^i_j$ are absolutely continuous with respect to the Lebesgue measure and
that the function $\varphi^{\gamma_0}$ (see Theorem \ref{parodontax}(ii)) is integrable with respect to
the measure $|\mu^i_j|$ for any $i,j=1,\ldots,m$.

\begin{prop}
\label{giampaolo-co}
Each measure $|\mu^i_j|$ is absolutely continuous with respect to the Lebesgue measure. Moreover, for any $i,j=1, \ldots, m$, $\varphi^{\gamma_0}\in L^1(\Rd,|\mu^i_j|)$ and $\|\varphi^{\gamma_0}\|_{L^1(\Rd;|\mu^i_j|)}\le C_2(\inf_{\Rd}\varphi)^{\gamma_0+1/\gamma}$, where $\gamma_0$ and $C_2$ are defined in the statement of Theorem
$\ref{parodontax}$.
\end{prop}

\begin{proof}
We fix $i,j\in\{1,\ldots,m\}$ and split the proof into two steps. In the first one we prove the absolutely continuity of $|\mu^i_j|$ with respect to
the Lebesgue measure. Then, in Step 2, we prove that the function $\varphi^{\gamma_0}$  is in $L^1(\Rd,|\mu^i_j|)$.

{\em Step 1.} We denote by $x_0\in\Rd$ the point where $\varphi$ attains its minimum value and introduce the family of measures $\{r_{ij}(t,x_0,dy): t>0\}$, defined by
\begin{align}
r_{ij}(t,x_0,B):=\frac{1}{t}\int_0^tp_{ij}(s,x_0,B)ds
\label{r-ij}
\end{align}
for any $t>0$, any Borel set $B\subset\Rd$ and $i,j=1,\ldots,m$. Note that $p_{ij}(s,x_0,B)=(\T(s)(f\e_j))_i(x_0)$ for any $s$, $i$, $j$ and $B$ as above. Since
the semigroup $\T(t)$ is strong Feller, the function $p_{ij}(\cdot,x_0,B)$ is continuous in $(0,+\infty)$ and bounded, due to estimate
\eqref{giroitalia-1}. Hence, the integral in the right-hand side of \eqref{r-ij} is well defined. Moreover, each $r_{ij}(t,x_0,dy)$ is a finite measure. Indeed,
we can write
\begin{eqnarray*}
|p_{ij}|(s,x_0,\Rd)=\sup\bigg\{\int_{\Rd}\zeta p_{ij}(s,x_0,dy): \zeta\in C_c(\Rd),\,\|\zeta\|_{\infty}\le 1\bigg\}
\end{eqnarray*}
(see e.g., \cite[Proposition 1.43]{AFP}) and, again by \eqref{giroitalia-1},
the function $s\mapsto\int_{\Rd}\zeta p_{ij}(s,x_0,dy)=(\T(s)(\zeta {\bf e}_j))_i(x_0)$ is bounded in $(0,+\infty)$. Therefore,
\begin{eqnarray*}
|r_{ij}|(t,x_0,\Rd)\le\frac{1}{t}\int_0^t|p_{ij}|(s,x_0,\Rd)ds\le C_2(\varphi(x_0))^{\frac{1}{\gamma}},\qquad\;\,t>0.
\end{eqnarray*}
In view of Theorem \ref{succhinopera}, for any $t>0$ and $f\in C_b(\Rd)$ it holds that
\begin{align*}
\int_{\Rd}f r_{ij}(t,x_0,dy)=\frac{1}{t}\int_0^t(\T(s)f\e_j)_i(x_0)ds
\end{align*}
and the right-hand side of the previous formula converges to $\int_{\Rd}fd\mu^i_j$ as $t$ tends to $+\infty$.
Hence, $r_{ij}(t,x_0,dy)$ weakly$^*$ converges to $\mu^i_j$ as $t$ tends to $\infty$.

Now, we claim that $|r_{ij}|(t,x_0,\Omega)$ converges to $|\mu^i_j|(\Omega)$ as $t$ tends to $+\infty$, for any open set $\Omega\subset\Rd$.
For this purpose, we fix a sequence $(t_n)$ diverging to $+\infty$ such that $|r_{ij}|(t_n,x_0,\Omega)$ admits limit as $n$ tends to $+\infty$. Again by
\cite[Proposition 1.43]{AFP}, we can determine a sequence $(\zeta_n)\subset C_c(\Omega)$ with $\|\zeta_n\|_\infty\le 1$ such that
\begin{align}
|r_{ij}|(t_n,x_0,\Omega) \leq\int_{\Rd}\zeta_n(y)r_{ij}(t_n,x_0,dy)+\frac{1}{n}=
\frac{1}{t_n}\int_0^{t_n}(\T(s)(\zeta_n\e_j))_i(x_0)ds+\frac{1}{n}
\label{borrini}
\end{align}
for any $n\in\N$. Now, we observe that if $t_n>1$ then we can write
\begin{align}
\frac{1}{t_n}\int_0^{t_n}(\T(s)(\zeta_n\e_j))_i(x_0)ds=&\frac{1}{t_n}\int_0^1(\T(s)(\zeta_n\e_j))_i(x_0)ds\notag\\
&+\frac{1}{t_n}\int_0^{t_n}(\T(s)\T(1)(\zeta_n\e_j))_i(x_0)ds\notag\\
&-\frac{1}{t_n}\int_{t_n-1}^{t_n}(\T(s+1)(\zeta_n\e_j))_i(x_0)ds.
\label{polidoro}
\end{align}
Since $\|\T(s)(\zeta_n\e_j)\|_{\infty}\le e^{\beta}\|\zeta_n\|_{\infty}\le e^{\beta}$ for any $s\in [0,1]$, the first term in the right-hand side
of \eqref{polidoro} vanishes as $n$ tends to $+\infty$. Similarly, $|(\T(s+1)(\zeta_n\e_j))_i(x_0)|\le C\sqrt{\varphi(x_0)}$ for any $s>0$, by \eqref{giroitalia}.
Hence, also the third term in the right-hand side of \eqref{polidoro} vanishes as $n$ tends to $+\infty$.
As far as the second term in the right-hand side of \eqref{polidoro} is concerned, we observe that the interior Schauder estimates in Theorem \ref{thm-A2} show that there exists an increasing sequence $(n_k)\subset\N$ such that $\T(1)(\zeta_{n_k}\e_j)$ converges
locally uniformly in $\Rd$ to bounded and continuous function $\g$. This result and Proposition \ref{carovana} imply that $(\T(s)(\T(1)(\zeta_{n_k}\e_j)-\g))_i(x_0)$
converges to $0$ uniformly in $(0,+\infty)$ as $k$ tends to $+\infty$.
Therefore,
\begin{eqnarray*}
\lim_{k\to +\infty}\frac{1}{t_{n_k}}\int_0^{t_{n_k}}(\T(s)(\T(1)(\zeta_{n_k}\e_j)-\g))_i(x_0)ds=0
\end{eqnarray*}
and from \eqref{polidoro} we conclude that
\begin{align*}
\lim_{k\to+\infty}\frac{1}{t_{n_k}}\int_0^{t_{n_k}}(\T(s)(\zeta_{n_k}\e_j))_i(x_0)ds=&
\lim_{k\to+\infty}\frac{1}{t_{n_k}}\int_0^{t_{n_k}}(\T(s)\g)_i(x_0)ds\\
=&\sum_{j=1}^m\int_{\Rd}g_jd\mu^i_j.
\end{align*}
We claim that $\sum_{j=1}^m\int_{\Rd}g_jd\mu^i_j\le |\mu^i_j|(\Omega)$. For this purpose we use the invariance property of the family $\{\mu^i_j: j=1,\ldots,m\}$ to write
\begin{eqnarray*}
\sum_{h=1}^m\int_{\Rd}(\T(1)(\zeta_{n_k}\e_j))_hd\mu_h^i=\int_{\Rd}\zeta_{n_k}d\mu^i_j
\end{eqnarray*}
for any $k\in\N$. Hence, by dominated convergence we obtain
\begin{align*}
\sum_{h=1}^m\int_{\Rd}g_hd\mu_h^i=&\lim_{k\to +\infty}\sum_{h=1}^m\int_{\Rd}(\T(1)(\zeta_{n_k}\e_j))_hd\mu_h^i
=\lim_{k\to +\infty}\int_{\Rd}\zeta_{n_k}d\mu^i_j\\
\le &\limsup_{k\to +\infty} \int_{\Rd}|\zeta_{n_k}|d|\mu^i_j|\le |\mu^i_j|(\Omega).
\end{align*}

Now, we are almost done. Indeed, writing \eqref{borrini} with $n_k$ replacing $n$ and letting $k$ tend to $+\infty$, we deduce that
\begin{align*}
\lim_{n\to +\infty}|r_{ij}|(t_n,x_0,\Omega)
=\lim_{k\to +\infty}|r_{ij}|(t_{n_k},x_0,\Omega)\le |\mu^i_j|(\Omega).
\end{align*}
The arbitrariness of the sequence $(t_n)$ yields that $\limsup_{t\to +\infty}|r_{ij}|(t,x_0,\Omega)\le |\mu^i_j|(\Omega)$.
On the other hand, $|\mu^i_j|(\Omega)\le\liminf_{t\to +\infty}|r_{ij}|(t,x_0,\Omega)$. Indeed,
since $r_{ij}(t,x_0,dy)$ weakly$^*$ converges to $\mu^i_j$ as $t$ tends to $+\infty$, we can write
\begin{eqnarray*}
\int_{\Rd}\zeta d\mu^i_j=\lim_{t\to +\infty}\int_{\Rd}\zeta r_{ij}(t,x_0,dy)\le
\liminf_{t\to +\infty}|r_{ij}|(t,x_0,\Omega),
\end{eqnarray*}
due to the fact that $\int_{\Rd}\zeta r_{ij}(t,x_0,dy)\le |r_{ij}|(t,x_0,\Omega)$ for any $t>0$. We have so proved that
$\limsup_{t\to +\infty}|r_{ij}|(t,x_0,\Omega)\le |\mu^i_j|(\Omega)\le \liminf_{t\to +\infty}|r_{ij}|(t,x_0,\Omega)$
i.e., $|r_{ij}|(t,x_0,\Omega)$ converges to $|\mu^i_j|(\Omega)$ as $t\to +\infty$.

It is now straightforward to show that $|r_{ij}|(t,x_0,C)$ converges to $|\mu^i_j|(C)$ as $t$ tends to $+\infty$ also when $C$ is a closed set.

Let us prove that each measure $\mu_j^i$ is absolutely continuous with respect to the Lebesgue measure. For this purpose, we fix a Borel set $B\subset \Rd$ with null Lebesgue measure and, for any $\varepsilon>0$, we denote by $K_\varepsilon\subset B$ a compact set such that
$|\mu^i_j|(B\setminus K_{\varepsilon})\leq \varepsilon$ (see e.g., \cite[Theorem 2.8]{maggi}).
Since each measure $p_{ij}(t,0,dy)$ is absolutely continuous with respect to the Lebesgue measure
(see \cite[Theorem 3.3]{AALT}), $r_{ij}(t,x_0,dy)$ is absolutely continuous with respect to the Lebesgue measure, as well, for any $t>0$.
Hence, $|r_{ij}|(t,x_0,K_\varepsilon)=0$ for any $t>0$ and $|\mu^i_j|(K_\varepsilon)=0$. Splitting $B$ into the union of $K_{\varepsilon}$ and $B\setminus K_{\varepsilon}$, we thus conclude that
$|\mu^i_j|(B)\le \varepsilon$ and the arbitrariness of $\varepsilon>0$ shows that $|\mu^i_j|(B)=0$ and we are done.

{\em Step 2.} To begin with, we claim that \eqref{marito} can be extended to any bounded Borel measurable function $\f:\Rd\to\Rm$.
For this purpose, we approximate any such function $\f$ by a bounded sequence $(\f_n)\subset C_b(\Rd;\Rm)$ which converges to $\f$ almost everywhere in $\Rd$. By the proof of Theorem \ref{parodontax},
$(\T(\cdot)\f_n)$ converges to $\T(\cdot)\f$ pointwise in $\Rd$. Since the measures $|\mu_j^i|$ are absolutely continuous with respect to the Lebesgue measure,
writing \eqref{caricatore} with $\f$ and $\mu_j$ being replaced by $\f_n$ and $\mu_j^i$, respectively, and letting $n$ tend to $+\infty$ (taking \eqref{rettore} into account), we get \eqref{caricatore} in its full generality.
Now, as in \eqref{polidoro} we write
\begin{align*}
\frac{1}{t}\int_0^t(\T(s)\f)_i(x_0)ds=o(1)+\frac{1}{t}\int_0^t(\T(s)\T(1)\f)_i(x_0)ds
\end{align*}
as $t$ tends to $+\infty$. The strong Feller property of the semigroup $\T(t)$ and \eqref{marito} yield that
\begin{align*}
\lim_{t \to +\infty}\frac{1}{t}\int_0^t(\T(s)\f)_i(x_0)ds= \sum_{j=1}^m \int_{\Rd}(\T(1)\f)_j d\mu^i_j= \sum_{j=1}^m \int_{\Rd}f_j d\mu^i_j
\end{align*}
which proves the claim.

By Riesz's theorem (see e.g., \cite[Theorem 4.7]{maggi}), there exists a measurable function $g_{ij}$ such that $|g_{ij}|=1$ everywhere in $\Rd$ such that
\begin{eqnarray*}
\int_{\Rd}f d\mu^i_j=\int_{\Rd}f g_{ij}d|\mu^i_j|,\qquad\;\,f\in C_c(\Rd).
\end{eqnarray*}
Since $|\mu^i_j|$ is absolutely continuous with respect the Lebesgue measure, using the dominated convergence theorem we can extend the above equality to any $f\in B_b(\Rd)$. Equivalently, we can write
\begin{eqnarray*}
\int_{\Rd}\frac{f}{g_{ij}}d\mu^i_j=\int_{\Rd}f d|\mu^i_j|,\qquad\;\,f\in B_b(\Rd).
\end{eqnarray*}
From all above and \eqref{giroitalia-1}, we deduce that
\begin{align*}
\int_{\Rd}\vartheta_n \varphi^{\gamma_0} d|\mu^i_j|=&\int_{\Rd}\frac{\vartheta_n\varphi^{\gamma_0}}{g_{ij}}d\mu^i_j\\
=&\lim_{t\to +\infty}\frac{1}{t}\int_0^t\bigg (\T(s)\bigg (\frac{\vartheta_n\varphi^{\gamma_0}}{g_{ij}}\e_j\bigg )\bigg )_i(x_0)ds\\
\le &\limsup_{t\to +\infty}\frac{1}{t}\int_0^t\bigg |\bigg (\T(s)\bigg (\frac{\vartheta_n\varphi^{\gamma_0}}{g_{ij}}\e_j\bigg )\bigg )_i(x_0)\bigg |ds
\le C_2(\varphi(x_0))^{\gamma_0+\frac{1}{\gamma}},
\end{align*}
where $(\vartheta_n)$ is a standard sequence of cut-off functions. Thus, Fatou lemma yields the assertion.
This concludes the proof.
\end{proof}

\medskip

As an important consequence of the previous proposition we can prove the following characterization of the evolution systems of measures $\{\mu_j: j=1,\ldots,m\}$
 such that $\varphi^{\gamma_0}\in L^1(\Rd,|\mu_j|)$ for any $j=1,\ldots,m$.

\begin{thm}
\label{chimica}
Let $\{\mu_j: j=1,\ldots,m\}$ be a family of Borel measures such that $\varphi^{\gamma_0}\in L^1(\Rd,|\mu_j|)$ for any $j=1,\ldots,m$. Then $\{\mu_j: j=1,\ldots,m\}$ is a system of invariant measures for $\T(t)$ if and only if there exist real constants $c_1,\ldots,c_m$ such that
\begin{align}
\label{marione}
\mu_j=\sum_{i=1}^mc_i\mu^i_j, \qquad j=1,\ldots,m.
\end{align}
\end{thm}

\begin{proof}
To begin with, we observe that if the measures $\mu_j$ ($j=1,\ldots,m$) are defined by \eqref{marione} for some real constants $c_i$ $(i=1,\ldots,m)$, then
$\{\mu_j: j=1,\ldots,m\}$ is a system of invariant measures of $\T(t)$. Indeed, for any $\f\in C_b(\Rd;\Rm)$ and $t\geq 0$
\begin{align*}
\sum_{j=1}^m \int_{\Rd}f_jd\mu_j
=  \sum_{i,j=1}^mc_i\int_{\Rd}f_j d\mu^i_j
=  \sum_{i,j=1}^mc_i\int_{\Rd}(\T(t)f)_jd\mu^i_j
= \sum_{j=1}^m\int_{\Rd}(\T(t)f)_jd\mu_j.
\end{align*}

Let us now suppose that $\{\mu_j: j=1,\ldots,m\}$ is a system of invariant measures of $\T(t)$. Then,
\begin{align}
\label{programma}
\sum_{j=1}^m\int_{\Rd}f_jd\mu_j=\sum_{j=1}^m\int_{\Rd}(\T(t)\f)_jd\mu_j,\qquad\;\,t>0,\;\,\f\in C_b(\Rd;\Rm).
\end{align}
Integrating both the sides of \eqref{programma} between $0$ and $t$ and then dividing by $t$, we get
\begin{align}
\label{programma-1}
\sum_{j=1}^m\int_{\Rd}f_jd\mu_j=\sum_{j=1}^m\int_{\Rd}(\bm{\mathcal{{P}}}_t\f)_jd\mu_j,
\end{align}
where the operator $\bm{\mathcal P}_t$ has been introduced in \eqref{marito}. Since $\gamma_0>1/\gamma$, by Theorem \ref{parodontax}(i), we
can estimate $|(\bm{\mathcal P}_t\f)(x)|\le C_0(\varphi(x))^{\gamma_0}$ for any $t>0$ and $x\in\Rd$, where $C_0$ is the constant
in \eqref{giroitalia}. Since $\varphi^{\gamma_0}\in L^1(\Rd,|\mu_j|)$ for any $j=1,\ldots,m$
and $(\bm{\mathcal P}_t\f)_j$ converges to $\sum_{k=1}^m\int_{\Rd}f_kd\mu_k^j$ as $t$ tends to $+\infty$,
we can let $t$ tend to $+\infty$ in both sides of \eqref{programma-1} and conclude that
\begin{align*}
\sum_{j=1}^m\int_{\Rd}f_jd\mu_j=\sum_{k,j=1}^m\int_{\Rd}f_k\mu_j(\Rd)d\mu_k^j
= \sum_{k=1}^m\int_{\Rd}f_kd\lambda_k,
\end{align*}
where $d\lambda_k:=\sum_{j=1}^m\mu_j(\Rd)d\mu_k^j$ for any $k=1,\ldots,m$.
Fix $j\in \{1,\ldots,m\}$. Taking $\f=f\e_j$ in the above formula, reveals that
\begin{align*}
\int_{\Rd}f_jd\mu_j=\int_{\Rd}f_jd\lambda_j,
\end{align*}
which means that the measures $\lambda_j$ and $\mu_j$ coincide in $C_b(\Rd)$, for any $j=1,\ldots,m$.
Riesz theorem implies that $\lambda_j$ and $\mu_j$ coincide on the Borel sets of $\Rd$, for any $j=1,\ldots,m$.
Hence, formula \eqref{marione} holds true with $c_i=\mu_i(\Rd)$.
\end{proof}

\medskip

Proposition \ref{giampaolo-co} and the equivalence between $\mu$ and the Lebesgue measure yield immediately that $|\mu^i_j|$ is absolutely continuous with respect to $\mu$ for any $i,j=1,\ldots,m$. Next theorem provides a more refined result on the density of $|\mu^i_j|$ with respect to the measure $\mu$.

\begin{thm}\label{capitano}
For any $i,j=1,\ldots,m$, let $g_{ij}$ be the density of the measure $|\mu^i_j|$ with respect to the measure $\mu$. Then,
$g_{ij}\in L^{r_0}(\Rd,\mu)\cap W^{1,q}_{\rm loc}(\Rd)$ for any $q<+\infty$, where $r_0=\min\{\gamma,p_0'\}$, $\gamma$ is the constant appearing in Hypothesis $\ref{hyp_base}(v)$ and
$p_0'$ is the exponent conjugate to $p_0$.
\end{thm}

\begin{proof}
To begin with, let us prove that each function $g_{ij}$ belongs to $L^{r_0}(\Rd,\mu)$. For this purpose, we fix $i,j\in\{1,\ldots,m\}$,
$\f\in C_b(\Rd;\Rm)$ and recall that, up to a subsequence,
$\T(\cdot)\f$ is the pointwise limit of the sequence $(\uu_n)$, where $\uu_n$ is implicitly defined by the equation
\begin{align}
u_{n,i}(t,x)=(T(t)f_i)(x)+\int_0^t(T(t-s)w_{n,i}(s,\cdot))(x)ds
\label{cibalgina}
\end{align}
for any $t>0$ and $x\in\Rd$, and $w_{n,i}=\sum_{j=1}^d\sum_{i,h=1}^m(B_j)_{h,m}D_ju_{n,h}$ (see the proof of Theorem \ref{parodontax}). In particular, the sequence $(\uu_n)$ is bounded in each strip $[0,T]\times\Rd$.

Integrating both sides of \eqref{cibalgina} in $\Rd$ with respect to $\mu$ we get
\begin{align}
\int_{\Rd}u_{n,i}(t,\cdot)d\mu=&\int_{\Rd}T(t)f_id\mu+\int_{\Rd}d\mu\int_0^tT(t-s)w_{n,i}(s,\cdot)ds\notag\\
=&\int_{\Rd}f_id\mu+\int_0^tds\int_{\Rd}T(t-s)w_{n,i}(s,\cdot)d\mu\notag\\
=&\int_{\Rd}f_id\mu+\int_0^tds\int_{\Rd}w_{n,i}(s,\cdot)d\mu.
\label{simone-1}
\end{align}
Here, we have used the continuity of the function $(s,x)\mapsto (T(t-s)w_{n,i}(s,\cdot))(x)$, together with the estimate $|(T(t-s)w_{n,i}(s,\cdot))(x)|\le c_ns^{-1/2}\|f\|_{\infty}$ for any $s\in (0,t)$, $x\in\Rd$, to change the order of integration, and the invariance property of the measure $\mu$.

Now, we distinguish the cases $\gamma>p_0'$ and $\gamma\le p_0'$. In the first case, we use Hypothesis \ref{hyp_base}(ii), estimate \eqref{capitanfuturo} and the invariance property of $\mu$ to deduce that
\begin{align}
\int_{\Rd}|w_{n,i}(s,\cdot)|d\mu\le & c\int_{\Rd}\psi |J_x\T(s)\f|d\mu\notag\\
\le &c\|\psi\|_{L^{p_0'}(\Rd,\mu)}\| |J_x\T(s)\f| \|_{L^{p_0}(\Rd,\mu)}\notag\\
\le & ce^{\sigma_{p_0}s}\|\varphi\|_{L^1(\Rd,\mu)}^{\frac{p_0-1}{p_0}}(1\vee s^{-\frac{1}{2}})\|T(s)|\f|^{p_0}\|_{L^1(\Rd,\mu)}^{\frac{1}{p_0}}\notag\\
=& ce^{\sigma_{p_0}s}(1\vee s^{-\frac{1}{2}})\||\f|\|_{L^{p_0}(\Rd,\mu)}
\label{simone-2}
\end{align}
for any $s>0$. On the other hand, if $\gamma\le p_0'$, arguing similarly, we estimate
\begin{align}
\int_{\Rd}|w_{n,i}(s,\cdot)|d\mu\le &c\|\psi\|_{L^{\gamma}(\Rd,\mu)}\| |J_x\T(s)\f| \|_{L^{\frac{\gamma}{\gamma-1}}(\Rd,\mu)}\notag\\
\le & ce^{\sigma_{p_0}s}(1\vee s^{-\frac{1}{2}})\||\f|\|_{L^{\frac{\gamma}{\gamma-1}}(\Rd,\mu)}
\label{simone-2a}
\end{align}
for any $s>0$. From \eqref{simone-1}-\eqref{simone-2a} we can infer that
\begin{align}
\int_{\Rd}u_{n,i}(t,\cdot)d\mu
\le & \int_{\Rd}f_id\mu+c\||\f|\|_{L^{r_0'}(\Rd,\mu)}
\int_0^{+\infty}e^{\sigma_{p_0}s}(1\vee s^{-\frac{1}{2}})ds.
\label{simone-3}
\end{align}
Letting $n$ tend to $+\infty$ in \eqref{simone-3} we conclude that
\begin{align}
\int_{\Rd}(\T(t)\f)_i d\mu
\le & \int_{\Rd}f_id\mu+c\| |\f| \|_{L^{r_0'}(\Rd,\mu)}.
\label{simone-4}
\end{align}
Now, we let $t$ tend to $+\infty$ in \eqref{simone-4}. Taking \eqref{giroitalia} and the forthcoming Theorem \ref{venditti} into account, we can apply the dominated convergence theorem and deduce that
\begin{align}
\sum_{j=1}^m\int_{\Rd}f_jd\mu^i_j=
\int_{\Rd}\bigg (\sum_{j=1}^m\int_{\Rd}f_jd\mu^i_j\bigg )d\mu
\le \int_{\Rd}f_id\mu+c\| |\f| \|_{L^{r_0'}(\Rd,\mu)}.
\label{simone-5}
\end{align}

To go further, we extend \eqref{simone-5} to any $\f\in B_b(\Rd;\Rm)$ by approximating any such function $\f$ by a bounded sequence $(\f_n)\subset C_b(\Rd;\Rm)$ which converges to $\f$ almost everywhere (with respect to the Lebesgue measure and, hence, with respect to the measures $\mu^i_j$ and $\mu$) in $\Rd$. Writing
\eqref{simone-5} with $\f$ being replaced by the function $\f_n$ and letting $n$ tend to $+\infty$, by dominated convergence we obtain that $\f$ satisfies
\eqref{simone-5} as well.

Now, we are almost done. Indeed, take $f\in B_b(\Rd)$ and let $A_{ij}^+$ be the set where the positive part of $\mu^i_j$ is concentrated.
Writing \eqref{simone-5} with $\f= f\chi_{A_{ij}^+}\e_k$ gives
\begin{align*}
\int_{\Rd}fd(\mu^i_j)^{+}
\le &\delta_{ij}\int_{A_{ij}^+}fd\mu+c\|f\chi_{A_{ij}^+}\|_{L^{r_0'}(\Rd,\mu)}\notag\\
\le & \int_{\Rd}|f| d\mu+c\|f\|_{L^{r_0'}(\Rd,\mu)}\le (1+c)\|f\|_{L^{r_0'}(\Rd,\mu)}.
\end{align*}

This shows that the operator $f\mapsto \int_{\Rd}fd(\mu^i_j)^{+}$ can be (uniquely) extended to a bounded linear operator on $L^{r_0'}(\Rd,\mu)$ and the Riesz's representation theorem implies that there exists a nonnegative function $\phi^+_{ij}\in L^{r_0}(\Rd,\mu)$ such that
\begin{eqnarray*}
\int_{\Rd}fd(\mu^i_j)^{+}=\int_{\Rd}f\phi_{ij}^+d\mu,\qquad\;\,f\in L^{r_0'}(\Rd,\mu).
\end{eqnarray*}

Repeating the same arguments with $A_{ij}^+$ being replaced by the set $A_{ij}^-$, where the negative part of $\mu^i_j$
is concentrated, we can show that
\begin{eqnarray*}
\int_{\Rd}fd(\mu^i_j)^{-}=\int_{\Rd}f\phi^-_{ij}d\mu,\qquad\;\,f\in L^{r_0'}(\Rd,\mu)
\end{eqnarray*}
for some nonnegative function $\phi^-_{ij}\in L^{r_0}(\Rd,\mu)$. Since $g_{ij}=\phi^+_{ij}+\phi_{ij}^-$, we immediately conclude that $g_{ij}\in L^{r_0}(\Rd,\mu)$.

To conclude the proof, let us show that the function $g_{ij}$ belongs to $W^{1,q}_{\rm loc}(\Rd)$ for any $i,j=1,\ldots,m$ and any $q<+\infty$.
For this purpose, we use a bootstrap argument. We fix $r>0$, $\eta\in C_c^\infty(B(0,r))$ and $i,j\in\{1,\ldots,m\}$. Choosing $\f=\eta\e_j$ in \eqref{trasportino}
and observing that $d\mu_j^i=(\phi^+_{ij}-\phi_{ij}^-)d\mu:=h_{ij}d\mu$ for any $i,j=1,\ldots,m$, we get
\begin{align*}
\int_{\Rd}h_{ij}\rho \mathcal A \eta  dx
=-\sum_{h=1}^d\sum_{k=1}^m\int_{\Rd}(B_h)_{kj}D_h\eta (\phi^+_{ik}-\phi_{ik}^-)\rho dx.
\end{align*}
Let us estimate the right-hand side of the previous formula, which we denote by ${\mathcal I}_{ij}$. From the first part of the proof, we know that $\phi^{\pm}_{hk}\in L^2(\Rd,\mu)$ for any $h,k=1,\ldots,m$. Moreover, \cite[Corollary 2.9]{bogkroroc-00} implies that $\rho\in W^{1,p}_{\rm loc}(\Rd)$ for any $p \ge 1$ and so, in particular, $\rho$ is locally H\"older continuous. Since the entries of the matrices $B_h$ ($h=1,\ldots,d$) are locally bounded, we get
\begin{align*}
|{\mathcal I}_{ij}|
\leq c\max_{\substack{h=1,\ldots,d\\ k=1,\ldots,m}}&\|(B_h)_{kj}\sqrt{\rho}\|_{L^\infty(B(0,r))}\\
&\times \|h_{ik}\|_{L^{2}(B(0,r),\mu)}\|\nabla \eta\|_{L^2(B(0,r))}
\end{align*}
and, therefore,
\begin{align}
\label{pitagora}
\bigg |\int_{\Rd}\mathcal A \eta d\mu^j_i\bigg |\leq c_r\|\nabla \eta\|_{L^2(B(0,r))}.
\end{align}

Now, we fix $f\in C_c^\infty(\Rd)$ and a smooth function $\psi_r$ such that $\chi_{B(0,r/2)}\le \psi_r\le \chi_{B(0,r)}$. Since
$\mathcal A(\zeta_1\zeta_2)=\zeta_1\mathcal A \zeta_2+ \zeta_2 \mathcal A \zeta_1+2\langle Q\nabla \zeta_1,\nabla \zeta_2\rangle$ for any pair of smooth functions $\zeta_1,\zeta_2$, we can estimate
\begin{align*}
\bigg |\int_{\Rd}(\mathcal A f) \psi_r h_{ik}\rho dx\bigg |=&\bigg |\int_{\Rd}(\mathcal A f) \psi_r d\mu^i_j\bigg |\\
\leq & \bigg |\int_{\Rd}\mathcal A(\psi_rf)d\mu^i_j\bigg |+\bigg |\int_{\Rd}\mathcal A\psi_rfd\mu^i_j\bigg |
+2\bigg |\int_{\Rd}\langle Q\nabla f,\nabla \psi_r\rangle d\mu^i_j\bigg |\\
=&\!: {\mathcal J}_1+{\mathcal J}_2+{\mathcal J}_3.
\end{align*}
We claim that
\begin{align}
\label{pisolino}
{\mathcal J}_1+{\mathcal J}_2+{\mathcal J}_3\leq c_r\|f\|_{W^{1,2}(\Rd)}.
\end{align}
Estimate \eqref{pitagora} shows that ${\mathcal J}_1\leq c_r\|\nabla(\psi_R f)\|_{L^2(B(0,r))}\leq c_r\|f\|_{W^{1,2}(\Rd)}$. As far as ${\mathcal J}_2$ and ${\mathcal J}_3$ are concerned, arguing as above we deduce that
\begin{align*}
&{\mathcal J}_2\leq \|(\mathcal A\psi_r)\sqrt{\rho}\|_{L^\infty(B(0,r))}\|h_{ij}\|_{L^{2}(B(0,r),\mu)}\|f\|_{L^2(\Rd)},\\[1mm]
&{\mathcal J}_3\le 2\||Q\nabla\psi_r|\sqrt{\rho}\|_{L^\infty(B(0,r))}\|h_{ij}\|_{L^{2}(B(0,r),\mu)}\|\nabla f\|_{L^2(\Rd)}.
\end{align*}
Estimate \eqref{pisolino} is so proved and, from \cite[Theorem D.1.4(ii)]{newbook}, we deduce that $h_{ij}\rho\psi_r\in W^{1,2}_{\rm loc}(\Rd)$ for any $i,j=1,\ldots, m$. Thus, $h_{ij}\rho$ belongs to $W^{1,2}(B(0,r/2))$. The arbitrariness of $r>0$ yields immediately that $h_{ij}\rho\in W^{1,2}_{\rm loc}(\Rd)$. Since $\rho \in W^{1,p}_{\rm loc}(\Rd)$ for any $p\ge 1$ and $\inf_{\Rd}\rho$ is positive, we can infer that $h_{ij}$ belongs to $W^{1,2}_{\rm loc}(\Rd)$.

Now, we can make the bootstrap argument work. By the Sobolev embedding theorem, $h_{ij}$ belongs to $L^{2^*}_{\rm loc}(\Rd)$ (and, hence, to $L^{2^*}_{\rm loc}(\Rd,\mu)$), where $1/2^*=1/2-1/d$. Thus, arguing as above, replacing $\|h_{ij}\|_{L^2(B(0,r),\mu)}$ by $\|h_{ij}\|_{L^{2^*}(B(0,r),\mu)}$ and $\|f\|_{W^{1,2}(\Rd)}$ by $\|f\|_{W^{1,(2^*)'}(\Rd)}$ in the estimate \eqref{pisolino} and applying again \cite[Theorem D.1.4(ii)]{newbook} we can infer that $g_{ij}$ belongs to $W^{1,2^*}_{\rm loc}(\Rd)$. Iterating this procedure, in a finite number of steps we get that $h_{ij}\in W^{1,p}_{\rm loc}(\Rd)$ for some $p>d$. We are almost done. Indeed, again by the Sobolev embedding theorem we deduce that $g_{ij}\in L^q_{\rm loc}(\Rd)$ for any $q\le +\infty$. Hence, we can write estimate \eqref{pisolino} with $\|f\|_{W^{1,2}(\Rd)}$ being replaced by $\|f\|_{W^{1,q'}(\Rd)}$, for any $q'<+\infty$, and \cite[Theorem D.1.4(ii)]{newbook} allows us to conclude that $h_{ij}$ belongs to $W^{1,q}_{\rm loc}(\Rd)$ for any $q<+\infty$.

Finally, we observe that $\phi^+_{ij}$ coincides with the positive part of $h_{ij}$ (and, hence, $\phi^{-}_{ij}$ coincides with the negative part of $h_{ij}$) as it is immediately checked recalling that
$\phi^{+}_{ij}$ and $\phi^{-}_{ij}$ are nonnegative functions with disjoint supports. Since the positive and negative parts of a function in $W^{1,p}_{\rm loc}(\Rd)$ belong to $W^{1,p}_{\rm loc}(\Rd)$, we immediately conclude that
$g_{ij}=\phi^+_{ij}+\phi^{-}_{ij}$ belongs to $W^{1,p}_{\rm loc}(\Rd)$ as well.
\end{proof}

\medskip

To conclude this subsection, we consider the particular case where the measure $\mu$ is symmetrizing for the scalar semigroup $T(t)$, i.e.,
\begin{equation}\label{premium}
\int_{\Rd}\mathcal{A}f \,gd\mu=-\int_{\Rd}\langle Q \nabla f,\nabla g\rangle d\mu.
\end{equation}
for any $f\in H^{2}_{\rm loc}(\Rd)$, $g\in H^1_{\rm loc}(\Rd)$ such that $f$ or $g$ has compact support.

\begin{rmk}
\label{rem-4.1}
{\rm Sufficient conditions for \eqref{premium} to hold are provided in \cite{lorenzi-lunardi-elliptic} under the following additional assumptions on the coefficients
$q_{ij}$ and $b_j$ ($i,j=1,\ldots,d$): there exists a function $\Phi:\Rd\to\R$ such that
\begin{enumerate}[\rm(i)]
\item $Q^{-1}(\textrm{div}Q-{\bf b})=\nabla\Phi$ where $(\textrm{div}Q)_j:= \sum_{i=1}^d D_i q_{ij}$ for any $j=1, \ldots, d$;
\item $e^{-\Phi}\in L^1(\Rd)$;
\item there exists two positive constants $k_1$ and $k_2\in(0,1)$ such that
\begin{align}
&\langle Q(x)(J{\bf b}(x))^*\xi,\xi\rangle
+\langle Q(x)\xi,\nabla{\rm Tr}(Q(x)S)\rangle
-{\rm Tr}((\nabla (Q(x)\xi))Q(x)S) \notag \\
\leq &k_1|\sqrt{Q}(x)\xi|^2+k_2|\sqrt{Q}(x)S\sqrt{Q}(x)|^2
\label{kebab}
\end{align}
for any $x,\xi\in\Rd$ and any $d\times d$ symmetric matrix $S$.
\end{enumerate}
}
\end{rmk}

Let $\{\mu_j: j=1,\ldots,m\}$ be a system of invariant measures for $\T(t)$ which consist of measures absolutely continuous with respect to the Lebesgue measure.
Since $\mu$ is equivalent to the Lebesgue measure, there exist a vector valued function $\boldsymbol\rho$ such that each $\rho_i$ belongs to $L^1(\Rd,\mu)$ and $d\mu_i=\rho_i d\mu$. For $\boldsymbol\rho$ smooth enough, next theorem relates the invariance property of the family $\{\mu_j: j=1,\ldots,m\}$ to a first-order differential
equation that $\boldsymbol\rho$ has to satisfy.

\begin{thm}
Under Hypothesis $\ref{hyp_base}$, assume that $|Q(x)|\le c(1+|x|^2)$ for any $x \in \Rd$, that the map $x\mapsto |Q(x)|$ belongs to $L^1(\Rd,\mu)$ and that $\mu$ is symmetrizing for the scalar semigroup $T(t)$.
Further, let $\{\mu_i:i=1,\ldots,m\}$ be a family of Borel finite measures, absolutely continuous with respect to the Lebesgue measure. Suppose that
$\boldsymbol\rho$ solves the first-order differential equations $(Q\nabla\rho_i)_j-(B_j^*\boldsymbol\rho)_i=0$ in $\Rd$ for any $i,j=1,\ldots, m$, with $\rho_i\in L^2(\Rd,\mu)\cap W^{1,p}_{\rm loc}(\Rd)$ for some $p<+\infty$ and any $i$ as above. Then, $\{\mu_i:i=1,\ldots, m\}$ is a system of invariant measures for $\T(t)$.
\end{thm}

\begin{proof}
We split the proof into two steps.

{\em Step 1.} Here, we prove that the set
\begin{eqnarray*}
D_{\mu}(\A):=\{\f\in D_{\rm max}(\A):\, |\sqrt{Q}\nabla f_i|\in L^2(\Rd,\mu),\,\,i=1,\ldots, m\}
\end{eqnarray*}
is a core for $(\A,D_{\max}(\A))$ with respect to the mixed topology\footnote{i.e.,
for any $\f\in D_{\max}(\A)$ there exists a sequence $(\f_n)\subset D_{\mu}(\A)$ such that $\sup_{n\in\N}(\|\f_n\|_{\infty}+\|\A\f_n\|_{\infty}<+\infty)$,
$\f_n$ and $\A\f_n$ converge to $\f$ and $\A\f$, respectively, locally uniformly in $\Rd$ as $n$ tends to $+\infty$} of $C_b(\Rd;\Rm)$.
To begin with, we prove that $R(n)$ (see \eqref{Res}) preserves $D_{\mu}(\A)$ for any $n\in\N$. For this purpose, we observe that, in view of Proposition \ref{bundes}, we can estimate
\begin{align*}
|\sqrt{Q}\nabla(R(n)\f)_i|^2=&\left\langle Q\int_0^{+\infty}e^{-nt}\nabla(\T(t)\f)_i(\cdot)dt,\int_0^{+\infty}e^{-nt}\nabla(\T(t)\f)_i(\cdot)dt\right\rangle\\
\le & |Q|\bigg |\int_0^{+\infty}e^{-nt}\nabla(\T(t)\f)_i(\cdot)dt\bigg |^2\\
\le &c|Q|\|\f\|_{\infty}\bigg |\int_0^{+\infty}e^{-nt}(1\vee t^{-\frac{1}{2}})dt\bigg |^2\le c|Q|\|\f\|_{\infty}.
\end{align*}
Since $|Q|\in L^1(\Rd,\mu)$ we deduce that $R(n)\f\in D_{\mu}(\A)$ for any $\f\in C_b(\Rd;\Rm)$.

Now, for any $\g\in C_b(\Rd;\Rm)$ and $n\in\N$ such that\footnote{Here, $[\beta]$ denotes the integer part of $\beta$.} $n>[\beta]$ (see \eqref{rettore}), we consider the function
$nR(n)\g$. Note that $\|R(n)\g\|_{C_b(\Rd;\Rm)}\le (n-\beta)^{-1}$ for any $n$ as above. Moreover, $nR(n)\g$ converges to $\g$ locally uniformly in $\R$ as $n$ tends to $+\infty$.
 This is clear if $\g$ belongs to $D_{\max}(\A)$. Indeed, we can split $nR(n)\g=\g-R(n)\A\g$ for any $n$ and, by the above estimate, $R(n)\A\g$ vanishes uniformly in $\Rd$ as $n$ tends to $+\infty$. Suppose that $\g\in C_b(\Rd;\Rm)$. Since $C^{\infty}_c(\Rd;\Rm)\subset D_{\max}(\A)$, we can determine a sequence $(\g_m)\subset D_{\max}(\A)$, bounded with respect to the sup-norm, which converges to $\g$ locally uniformly in $\Rd$. We split
\begin{align*}
\|nR(n)\g-\g\|_{C_b(B(0,r);\Rm)}\le &\|nR(n)(\g-\g_m)\|_{C_b(B(0,r);\Rm)}+\|\g_m-\g\|_{C_b(B(0,r);\Rm)}\\
&+\|nR(n)\g_m-\g_m\|_{C_b(B(0,r);\Rm)}\\
\le & \|\T(\cdot)(\g-\g_m)\|_{C_b((0,+\infty)\times B(0,r);\Rm)}\\
&+\|\g_m-\g\|_{C_b(B(0,r);\Rm)}\\
&+\|nR(n)\g_m-\g_m\|_{C_b(B(0,r);\Rm)}
\end{align*}
for any $m\in\N$, $n>[\beta]$ and $r>0$. Taking Proposition \ref{carovana} into account, we can let first $m$ and then $n$ tend to $+\infty$ in the first and last side of the previous chain of inequalities and conclude that $nR(n)\g$ converges to $\g$ locally uniformly in $\Rd$.

Given $\f\in D_{\max}(\A)$, the sequence $(\f_n)$ we are looking for can be defined by setting $\f_n=nR(n)\f$ for any $n>[\beta]$.

{\em Step 2.} In view of Proposition \ref{prop-benevento}, to prove that the system $\{\mu_i: i=1,\ldots,m\}$ is invariant for $\T(t)$ we need to show that
$\sum_{i=1}^m\int_{\Rd}(\A\f)_id\mu_i=0$ for any $\f\in D_{\max}(\A)$. By Step 1, we can limit ourselves to proving that the previous formula holds true for any $\f\in D_{\mu}(\A)$.
So, let us fix one such function $\f$ and let $(\vartheta_n)\in C_c^\infty(\Rd;\Rí)$ be a sequence of cut-off functions such that $\chi_{B(0,n)}\le\vartheta_n\leq \chi_{B(0,n+1)}$ and $|\nabla\vartheta_n|\leq cn^{-1}$ for any $n\in \N$. We set $\f_n:=\vartheta_n\f$ and using \eqref{premier} and (i) we obtain that
\begin{align}
\label{chico-1}
\sum_{i=1}^m\int_{\Rd}(\A\f_n)_id\mu_i
= & \sum_{i=1}^m\bigg [\int_{\Rd}({\mathcal A}f_{n,i})\rho_id\mu
+\sum_{j=1}^d\int_{\Rd}(B_j D_j \f_n)_i\rho_id\mu\bigg ] \notag \\
= & \sum_{i=1}^m\bigg [-\int_{\Rd}\langle Q\nabla f_{n,i}, \nabla\rho_i\rangle d\mu+\sum_{j=1}^d\int_{\Rd}(B_j D_j \f_n)_i\rho_id\mu\bigg ]\notag \\
= & \sum_{i=1}^m\sum_{j=1}^d\int_{\Rd}D_jf_{n,i}
[(B_j^*\boldsymbol\rho)_i-(Q\nabla\rho_i)_j]d\mu=0
\end{align}
for $n \in \N$. Now we show that the first side of \eqref{chico-1}
converges to $\sum_{i=1}^m\int_{\Rd}(\A\f)_id\mu_i$ as $n$ tends to $+\infty$. For this purpose, we observe that for any $n\in\N$ it holds that
\begin{align*}
(\A \f_n)_i=\vartheta_n (\A\f)_i+f_i\mathcal A\vartheta_n+\sum_{k=1}^m\sum_{j=1}^df_k(B_j)_{ik}D_j\vartheta_n+2\langle Q \nabla \vartheta_n,\nabla f_i\rangle.
\end{align*}
Integrating this formula over $\Rd$ with respect to $\mu_i$, summing up over $i$ from $1$ to $m$, using again \eqref{premium}, to write $\int_{\Rd}(\mathcal A\vartheta_n)f_i\rho_id\mu=-\int_{\Rd}\langle Q\nabla\vartheta_n,\nabla (f_i\rho_i)\rangle d\mu$, and the assumption on $\boldsymbol\rho$,  we get
\begin{align}\label{minest}
\sum_{i=1}^m\int_{\Rd}(\A \f_n)_id\mu_i=&
\sum_{i=1}^m\int_{\Rd}\vartheta_n (\A\f)_id\mu_i
+\sum_{i=1}^m\int_{\Rd}\langle Q\nabla\vartheta_n,\nabla f_i\rangle \rho_id\mu\notag\\
&+\int_{\Rd}\sum_{k=1}^m\sum_{j=1}^df_k(B_j)_{ik}D_j\vartheta_n d\mu_i
-\sum_{k=1}^m\int_{\Rd}f_k\langle Q\nabla\vartheta_n,\nabla\rho_k\rangle d\mu\notag\\
=&\sum_{i=1}^m\int_{\Rd}\vartheta_n (\A\f)_id\mu_i
+\sum_{k=1}^m\int_{\Rd}\langle \sqrt{Q}\nabla\vartheta_n,\sqrt{Q}\nabla f_k\rangle \rho_kd\mu.
\end{align}
By the dominated convergence theorem, the first term in the last side of \eqref{minest} converges to $\sum_{i=1}^m\int_{\Rd} (\A \f)_id\mu_i$ as $n$ tends to $+\infty$. In addition
\begin{align*}
\int_{\Rd}|\langle \sqrt{Q}\nabla\vartheta_n,\sqrt{Q}\nabla f_i\rangle \rho_i|d\mu
\leq &c \int_{\Rd\setminus B(0,n)}\frac{|\sqrt{Q}|}{n}|\sqrt{Q}\nabla f_i||\rho_i|d\mu\\
\le & c\|\sqrt{Q}\nabla f_i\|_{L^2(\Rd;\mu)}\int_{\Rd\setminus B(0,n)}\rho_i^2d\mu
\end{align*}
which vanishes as $n$ tends to $+\infty$, since $\rho_i\in L^2(\Rd,\mu)$.
As a byproduct, we conclude that $\sum_{i=1}^m\int_{\Rd}(\A \f_n)_id\mu_i$ converges to $\sum_{i=1}^m\int_{\Rd}(\A \f)_id\mu_i$ as $n$ tends to $+\infty$.
We have so proved that $\sum_{i=1}^m\int_{\Rd}(\A \f_n)_id\mu_i=0$. This completes the proof.
\end{proof}

\begin{example}
{\rm Here we assume $d=1$ and $m=2$. In this case $(\A\zeta)(x)=q(x)\boldsymbol{\zeta}''(x)+b(x)\boldsymbol{\zeta}'(x)+B\boldsymbol{\zeta}'(x)$ for any $x\in\R$, on smooth functions
$\boldsymbol{\zeta}:\R\to\R^2$. We suppose that $q,b$ and the entries of the matrix-valued function $B$ satisfy Hypotheses \ref{hyp_base} and the function
$x\mapsto-\log(q(x))+\int_0^x(q(s))^{-1}b(s)ds$ belongs to $L^1(\R)$. In this case, Remark \ref{rem-4.1} is satisfied and
\begin{eqnarray*}
\mu(dx)= \frac{c}{q(x)}\exp\Big(\int_0^x \frac{b(s)}{q(s)}ds\Big)dx
\end{eqnarray*}
for a suitable positive constant $c$. Note that condition (iii) in Remark \ref{rem-4.1} reduces to $q(x)b'(x)\xi^2\le k_1q(x)\xi^2+k_2(q(x))^2s^2$ for any $s,\xi,x\in\R$ and
some constants $k_1>0$ and $k_2\in (0,1)$, which is trivially satisfied since $b'<0$ in $\R$ due to Hypothesis \ref{hyp_base}(iv).

In order to compute a system of invariant measures associated to the vector-valued semigroup associated to $\A$ we further assume that $|q(x)|\le c(1+x^2)$ for any $x\in \R$, and
$B_{11}(x)+B_{12}(x)=B_{21}(x)+B_{22}(x)=:\beta(x)$ for any $x\in\R$. From Hypothesis \ref{hyp_base}(ii) it follows that the functions $B_{ij}$ ($i,j=1,2$) grow
at most linearly as $|x|$ tends to $+\infty$. We solve the system $q\boldsymbol\rho'=B^*\boldsymbol\rho$. Due to the above condition on the sum of the rows of $B$, we easily see that
$(\rho_1+\rho_2)'=q^{-1}(B_{11}+B_{12})(\rho_1+\rho_2)$. Hence,
\begin{eqnarray*}
\rho_1(x)+\rho_2(x)=c_1\exp\bigg (\int_0^x\frac{\beta(t)}{q(t)}dt\bigg ),\qquad\;\,x\in\R,
\end{eqnarray*}
for some positive constant $c_1$. Using this equation to write $\rho_2$ in terms of $\rho_1$ and replacing in the first equation of the above system, we easily see that
\begin{align*}
\rho_1(x) &=\exp\left(\int_0^x\frac{\gamma(t)}{q(t)}dt\right)\Bigg[c_2+c_1\displaystyle\int_0^x\exp\bigg (\int_0^t\frac{\beta(s)-\gamma(s)}{q(s)}ds\bigg )\frac{B_{21}(t)}{q(t)}dt\Bigg],
\end{align*}
and
\begin{align*}
\rho_2(x)=&\exp\bigg (\int_0^x\frac{\gamma(t)}{q(t)}dt\bigg )\Bigg[-c_2+c_1\exp\left (\int_0^x\frac{\beta(t)-\gamma(t)}{q(t)}dt\right )\\
&\phantom{\exp\bigg (\int_0^x\frac{\gamma(t)}{q(t)}dt\bigg )\Bigg[\;}-c_1\int_0^x\exp\bigg (\int_0^t\frac{\beta(s)-\gamma(s)}{q(s)}ds\bigg )\frac{B_{21}(t)}{q(t)}dt\Bigg]\\
=&\exp\bigg (\int_0^x\frac{\gamma(t)}{q(t)}dt\bigg )\Bigg[-c_2+c_1\exp\left (\int_0^x\frac{\beta(t)-\gamma(t)}{q(t)}dt\right )\\
&\phantom{\exp\bigg (\int_0^x\frac{\gamma(t)}{q(t)}dt\bigg )\Bigg[\;}-c_1\int_0^x\exp\bigg (\int_0^t\frac{\beta(s)-\gamma(s)}{q(s)}ds\bigg )\frac{\beta(t)-\gamma(t)}{q(t)}dt\\
&\phantom{\exp\bigg (\int_0^x\frac{\gamma(t)}{q(t)}dt\bigg )\Bigg[\;}+c_1\int_0^x\exp\bigg (\int_0^t\frac{\beta(s)-\gamma(s)}{q(s)}ds\bigg )\frac{B_{12}(t)}{q(t)}dt\Bigg]\\
=&\exp\bigg (\int_0^x\frac{\gamma(t)}{q(t)}dt\bigg )\Bigg[-c_2+c_1+c_1\int_0^x\exp\bigg (\int_0^t\frac{\beta(s)-\gamma(s)}{q(s)}ds\bigg )\frac{B_{12}(t)}{q(t)}dt\bigg ]
\end{align*}
for some positive constant $c_2\in \R$, where $\gamma=B_{11}-B_{21}$.

Now, we consider two concrete cases.

{\emph{Case 1}}. Here, we assume $q(x)=1$, $b(x)=-x$ for any $x\in\R$ and
\begin{eqnarray*}
B=\left(\begin{matrix}
0 \quad & -1\\
-1 \quad &0
\end{matrix}\right).
\end{eqnarray*}
This means that the scalar operator is the Ornstein-Uhlenbeck operator and the invariant measure of the associated Ornstein-Uhlenbeck operator $T(t)$ is the Gaussian measure $\mu(dx)=(2\pi)^{-1/2}e^{-x^2/2}dx$.
Hence, condition (ii) in Remark \ref{rem-4.1} is clearly satisfied.

From the above formulas for $\rho_1$ and $\rho_2$, we get
\begin{align*}
\rho_1(x)=a_1e^x+a_2e^{-x}, \qquad
\rho_2(x)=-a_1e^x+a_2e^{-x},\quad\;\,x\in\R,
\end{align*}
for any $a_1,a_2\in\R$.

{\emph{Case 2}}.
If $q(x)=1+x^2$, $b(x)=-b_0x(1+x^2)$ and $B_{ij}(x)=b_{ij}x$ for any $x\in\R$ and some positive constants $b_0$ and $b_{ij}$ ($i,j=1,2$)
such that
\begin{eqnarray*}
\bigg (\sum_{i,j=1}^2b_{ij}^2\bigg )^{\frac{1}{2}}+\left (\max_{1\le i,j\le 2}|b_{ij}|+1\right )^2<b_0,
\end{eqnarray*}
then Hypotheses \ref{hyp_base} are clearly satisfied. In particular, for any $h\in\N$, the function $\varphi$, defined by $\varphi(x)=(1+x^2)^h$ satisfies Hypothesis \ref{hyp_base}(iii). Moreover it is quite easy to show that the density of the invariant measure $\mu$ associated to the scalar semigroup $T(t)$ is the function
$x\mapsto (\pi e^{b/2}{\rm erfc}(2^{-1/2}b^{1/2}))(1+x^2)^{-1}\exp(-bx^2/2)dx$ for any $x\in\R$. Again, this implies that condition (ii) in Remark \ref{rem-4.1} holds true. It turns out that
\begin{align}
\left\{
\begin{array}{ll}
\displaystyle \rho_1(x)=(1+x^2)^{\frac{b_{11}-b_{21}}{2}}\bigg (c_2+\frac{c_1}{2} b_{21}\log(1+x^2)\bigg ),\\[3mm]
\displaystyle \rho_2(x)=(1+x^2)^{\frac{b_{11}-b_{21}}{2}}\bigg (c_1-c_2+\frac{c_1}{2}b_{12}\log(1+x^2)\bigg ) ,
\end{array}
\right.
\label{cineca}
\end{align}
for any $x\in\R$ and $c_1,c_2\in\R$, if $b_{12}=-b_{21}$, and
\begin{align*}
\left\{
\begin{array}{ll}
\displaystyle \rho_1(x)=(1+x^2)^{\frac{b_{11}-b_{21}}{2}}\bigg (c_2+c_1\frac{b_{21}}{b_{12}+b_{21}}(1+x^2)^{\frac{b_{12}+b_{21}}{2}}-c_1\frac{b_{21}}{b_{12}+b_{21}}\bigg ), \\[3mm]
\displaystyle
\rho_2(x)=(1+x^2)^{\frac{b_{11}-b_{21}}{2}}\bigg (-c_2+c_1\frac{b_{12}}{b_{12}+b_{21}}(1+x^2)^{\frac{b_{12}+b_{21}}{2}}+c_1\frac{b_{21}}{b_{12}+b_{21}}\bigg ), \\
\end{array}
\right.
\end{align*}
for any $x\in\R$ and $c_1,c_2\in\R$, otherwise.
Note that in both cases, the functions $\rho_1$ and $\rho_2$ belongs to $H^1_{\rm loc}(\R)\cap L^q(\R,\mu)$ for any $q<+\infty$.}
\end{example}

\begin{rmk}
{\rm We stress that, if $\rho_1$ and $\rho_2$ are given by \eqref{cineca} and $B$ is not diagonal, then, for any choice of the constants $c_1$ and $c_2$, at least one between $\mu_1=\rho_1 d\mu$ and $\mu_2=\rho_2 d\mu$ is not a positive measure. Indeed, suppose to fix the ideas that $b_{12}<0$. Then, $\rho_2$ is positive in a neighborhood of $+\infty$ if and only if $c_1\le 0$. Since $\rho_2(0)=c_1-c_2$, also $c_2$ should be non positive and
$c_1$, $c_2$ can not be both zero. If $c_2<0$, then $\rho_1(0)<0$, otherwise, if $c_2=0$, then $c_1<0$ and $\rho_2(0)=c_1<0$.}
\end{rmk}

\subsection{Asymptotic behaviour of the semigroup in $C_b(\Rd;\Rm)$}

As in the scalar case, the systems of invariant measures $\{\mu^i_j:j=1, \ldots, m\}$ ($i=1, \ldots, m$) provided by Theorem \ref{succhinopera} allow to characterize the asymptotic behaviour of $\T(t)$ in $C_b(\Rd;\Rm)$ as $t$ tends to $+\infty$.
\begin{thm}
\label{venditti}
For any $\f\in C_b(\Rd;\Rm)$ and $i=1, \ldots, m$, it holds that
\begin{eqnarray*}
\lim_{t \to +\infty}(\T(t)\f)_i=\sum_{j=1}^m\int_{\Rd}f_jd\mu^i_j,
\end{eqnarray*}
locally uniformly in $\Rd$.
\end{thm}
\begin{proof}
Fix $t>0$, $x \in \Rd$ and $\f\in C_b(\Rd;\Rm)$. Using the invariance property \eqref{caricatore} and taking into account that $\mu^i_j(\Rd)=\delta_{ij}$  for any $i,j=1,\ldots,m$ (see Subsection \ref{cmp}), we can write
\begin{align*}
\Bigg|(\T(t)\f)_i(x)-\sum_{j=1}^m\int_{\Rd} f_j(y)d\mu^i_j(y)\Bigg|
= \left|\sum_{j=1}^m\int_{\Rd}\big((\T(t)\f)_j(x)-(\T(t)\f)_j(y)\big)d\mu^i_j(y)\right|
\end{align*}
for any $i=1, \ldots,m$. Now, for any $t \ge 1$ set $B^t:=B(0,e^{-\sigma_2 t/2})$. Thus, using estimate \eqref{capitanfuturo}, recalling that $\gamma^{-1}\le\gamma_0$, $\varphi\ge 1$ in $\Rd$ and taking Proposition \ref{giampaolo-co} into account, we deduce that
\begin{align*}
&\bigg|\int_{\Rd}\big((\T(t)\f)_{j}(x) -(\T(t)\f)_j(y)\big)d\mu^i_j(y)\bigg| \\
\le &\int_{\Rd\setminus B^t}\left|\big((\T(t)\f)_{j}(x)-(\T(t)\f)_j(y)\big)\right|d\mu^i_j(y)\\
&+\int_{B^t}\left|\big((\T(t)\f)_{j}(x)-(\T(t)\f)_j(y)\big)\right|d\mu^i_j(y) \\
\le &c\|\f\|_{\infty}\int_{\Rd\setminus B^t}\varphi^{\frac{1}{\gamma}}d|\mu^i_j|
+ce^{\sigma_2 t}\int_{B^t}|x-y|d|\mu^i_j|(y) \\
\le &c\|\f\|_{\infty}\int_{\Rd\setminus B^t}\varphi^{\gamma_0}d|\mu^i_j|+c\left (e^{\sigma_2 t}|x|+e^{\frac{1}{2}\sigma_2t}\right )
\end{align*}
for any $i,j=1, \ldots, m$. Letting $t$ tend to $+\infty$ we obtain that the right-hand side tends to $0$
locally uniformly with respect to $x\in\Rd$ and this yields the claim.
\end{proof}

\appendix
\section{A priori estimates}
\label{sect-app}
\begin{thm}
Let $\uu$ belong to $W^{2,p}_{\rm loc}(\Rd;\R^m)$ for some $1<p<+\infty$. Then, for any pair of open bounded sets $\Omega_1$ and $\Omega_2$, $\Omega_1$ being compactly contained in $\Omega_2$, there exists a positive constant $c$, depending on $d,p,\Omega_1$, $\Omega_2$ the ellipticity constant of the operator $\A$ and the H\"older norm of its coefficients
 over $\Omega_2$, but independent of $\uu$, such that
\begin{equation}\label{lp_int}
\|\uu\|_{W^{2,p}(\Omega_1;\R^m)}\le c(\|\uu\|_{L^p(\Omega_2;\R^m)}+\|\A\uu\|_{L^p(\Omega_2;\R^m)}).
\end{equation}
\end{thm}
\begin{proof}
We divide the proof into two steps. In the first step we prove the claim when $\Omega_1= B(0,r)$ and $\Omega_2=B(0,2r)$, $r>0$.
In the second one, we complete the proof.

{\em Step 1.} For any $n\in \N$, we set $r_n=(2-2^{-n})r$. Clearly, $r_0=r$ and $r_\infty=2r$. We also set
\begin{eqnarray*}
\vartheta_n(x)=\vartheta\left(1+\frac{|x|-r_n}{r_{n+1}-r_n}\right),\qquad\;\,x\in\Rd,\;\,n\in\N,
\end{eqnarray*}
where $\vartheta\in C^\infty(\R)$ satisfies
$\chi_{(-\infty,1]}\leq\vartheta\leq\chi_{(-\infty,2]}$. Clearly, each function $\vartheta_n$ belongs to $C^\infty_c(\Rd)$, is such that $0\le \vartheta_n\le 1$, $\vartheta_n=1$ in $B(0,r_n)$ and $\supp(\vartheta_n)\subset B(0,r_{n+1})$. Moreover $\|\vartheta_n\|_{C_b^h(\Rd)}\le 2^{hn}c_r$ for $h=1,2$.

Applying classical global $L^p$-estimates to the functions $\vv_n:=\vartheta_n\uu$, which belong to $W^{2,p}(\Rd;\R^m)$
as well as the interpolative estimate $\| |J_x\vv_{n+1}| \|_{L^p(\Rd;\Rm)}\le c(\varepsilon\|D^2\vv_{n+1}\|_{L^p(\Rd;\Rm)}+\varepsilon^{-1}\|\vv_{n+1}\|_{L^p(\Rd;\Rm)})$, which holds true for any $\varepsilon>0$, we deduce that
\begin{align*}
\|\vv_n\|_{W^{2,p}(\Rd;\Rm)}&\le c_r(\|\vv_n\|_{L^p(\Rd;\Rm)}+\|\A\vv_n\|_{L^p(\Rd;\Rm)})\\
& \le c_r(\|\A\uu\|_{L^p(B(0,2r);\Rm)}+4^n\|\uu\|_{L^p(B(0,2r);\Rm)}\\
&\phantom{\le c_r(\,}+2^n\| |J_x\vv_{n+1}| \|_{L^p(\Rd;\Rm)})\\
&\le c_r(\|\A\uu\|_{L^p(B(0,2r);\Rm)}+(2^n\varepsilon^{-1}+4^n)\|\uu\|_{L^p(B(0,2r);\Rm)}\\
&\phantom{\le c_r(\;} +2^n\varepsilon\|\vv_{n+1}\|_{W^{2,p}(\Rd;\Rm)})
\end{align*}
for any $\varepsilon \in(0,1)$, where the constant $c$ depends also on $d$, $m$, $p$, the ellipticity constant of the operator $\A$ and the H\"older norm of its coefficients
 over $B(0,2r)$.
Choosing $\varepsilon=c^{-1}2^{-n-4}$ the previous inequality becomes
\begin{align*}
&\|\vv_n\|_{W^{2,p}(\Rd;\Rm)}-2^{-4}\|\vv_{n+1}|\|_{W^{2,p}(\Rd;\Rm)}\\
\le &c_r\|\A\uu\|_{L^p(B(0,2r);\Rm)}+4^{n+2}c_r\|\uu\|_{L^p(B(0,2r);\Rm)}.
\end{align*}
Multiplying both the terms by $2^{4n}$ and summing over $n$ from $0$ to $k\in \N$ we get
\begin{align*}
&\|\vv_0\|_{W^{2,p}(\Rd;\Rm)}\!-\!2^{-4k-4}\|\vv_{k+1}\|_{W^{2,p}(\Rd;\Rm)}\\
\le &c_r(\|\A\uu\|_{L^p(B(0,2r);\Rm)}+\|\uu\|_{L^p(B(0,2r);\Rm)}).
\end{align*}
Since $\|\vartheta_{k+1}\|_{C^2_b(\Rd)}\le 4^kc_r$, for any $k\in\N$, the second term in the left-hand side of the previous inequality vanishes as $k$ tends to $+\infty$ and this allows us to conclude the proof in this particular case, recalling that $\vv_0=\uu$ on $B(0,r)$.

{\em Step 2.}
Here, we complete the proof using a covering argument. Let $\Omega_1$ and $\Omega_2$ be as in the statement of the theorem. Further, fix $0<r<{\rm dist}(\Omega_1,\partial \Omega_2)$. By compactness we can cover $\Omega_1$ by a finite number of balls of radius $r$, i.e., there exist $x_1,\ldots, x_k$ in $\Omega_1$ such that $\overline{\Omega_1}\subset \bigcup_{i=1}^k(B(x_i,r))$. Due to the choice of $r$, $\bigcup_{i=1}^k(B(x_i,r))\subset\Omega_2$. By a translation, we can easily
extend estimate \eqref{lp_int} to balls centered at any point $x_0 \in \Rd$. Hence, we can write
\begin{align*}
\|\uu\|_{W^{2,p}(\Omega_1;\Rm)}&\le \sum_{i=1}^k\|\uu\|_{W^{2,p}(B(x_i,r);\R^m)}\\
&\le c\sum_{i=1}^k(\|\uu\|_{L^p(B(x_i,2r);\R^m)}+ \|\A\uu\|_{L^p(B(x_i,2r);\R^m)})\\
&\le c(\|\uu\|_{L^p(\Omega_2;\Rm)}+ \|\A\uu\|_{L^p(\Omega_2;\Rm)})
\end{align*}
and the claim is so proved.
\end{proof}

\begin{thm}[Theorem A.2, \cite{AALT}]
\label{thm-A2}
Let $\uu\in C^{1+\alpha/2,2+\alpha}_{\rm loc}((0,T]\times\Rd;\R^m)$
satisfy the differential equation
$D_t\uu = \A\uu+\g$ in $(0,T]\times\Rd$, for some $\g\in C^{\alpha/2,\alpha}_{\rm loc}((0,T]\times\Rd;\R^m)$ and $T>0$.
Then, for any $\tau\in(0,T)$ and any pair of bounded open sets $\Omega_1$ and $\Omega_2$, with $\Omega_1$ being compactly contained in $\Omega_2$, there exists a positive
constant $c$, independent of $\uu$, such that
\begin{align*}
&\|\uu\|_{C^{1+\alpha/2,2+\alpha}((\tau,T)\times\Omega_1;\R^m)}\notag\\
\le & c(\|\uu\|_{C_b((\tau/2,T)\times\Omega_2;\R^m)}+\|\g\|_{C^{\alpha/2,\alpha}((\tau/2,T)\times\Omega_2;\R^m)}).
\end{align*}
\end{thm}

\end{document}